\documentclass[11pt,twoside]{article}

\usepackage{mathrsfs,amsfonts,amsmath,amssymb}
\usepackage{latexsym}
\newtheorem{satz}{Theorem}
\newtheorem{proposition}[satz]{Proposition}
\newtheorem{theorem}[satz]{Theorem}
\newtheorem{lemma}[satz]{Lemma}

\newtheorem{corollary}[satz]{Corollary}
\newtheorem{remark}[satz]{Remark}
\newtheorem{example}[satz]{Example}

\def\eps{\varepsilon}
\def\_phi{\varphi}

\def\a{\alpha}
\def\d{\delta}
\def\la{\lambda}

\def\F{{\mathbb F}}

\def\t{\tilde}

\def\ov{\overline}

\def\C{{\mathbb C}}
\def\R{{\mathbb R}}
\def\E{\mathsf {E}}
\def\T{{\mathbb T}}

\def\Z_N{{\mathbb Z}_N}
\def\Z{{\mathbb Z}}

\def\f{{\mathbb F}}
\def\Gr{{\mathbf G}}

\def\Spec{{\rm Spec\,}}

\def\oT{{\rm T}}

\def\G{\Gamma}

\def\c{\circ}

\def\Cf{{\mathcal C}}

\def\T{\mathsf {T}}

\author{Shkredov I.D.}
\title{
Differences
of subgroups in subgroups
\footnote{
This work was supported by grant
Russian Scientific Foundation RSF 14--11--00433.}
}
\date{}
\begin{document}
\maketitle

\begin{center}
 Annotation.
\end{center}

{\it \small
    We prove, in particular, that if $A,\G \subset \f^*_p$, $|\G| < p^{3/4}$  are two
    %any
    arbitrary
    multiplicative subgroups
    %one has
    satisfying
    $A-A \subseteq \G \bigsqcup \{ 0\}$
    then
    %implies
    $|A| \ll |\G|^{1/3+o(1)}$.
    Also, we obtain that for any $\eps >0$ and
    %an arbitrary
    a sufficiently large
    subgroup $\G$ with $|\G| \ll p^{1/2-\eps}$ there is no representation $\G$ as $\G = A+B$, where $A$ is another subgroup, and $B$ is an arbitrary set, $|A|,|B|>1$.
    Finally, we study the number of collinear triples containing in a set of $\f_p$ and prove a "dual"\, sum--products estimate.
}
\\
%\\
%\\

\section{Introduction}
\label{sec:introduction}

Let $p$ be a prime number and $\f_p$ be the prime filed.
For two sets $A,B\subseteq \f_p$ we define its sumset as
$$
    A+B = \{ a+b ~:~ a\in A\,, b\in B \}
$$
and similarly its difference set, product set, and so on.
A set $S\subseteq \f_p$ is said to be {\it additively decomposable} if $S=A+B$, $|A|,|B|\ge 2$
and {\it primitive} otherwise.
S\'{a}rk\"{o}zy conjectured in \cite{Sarkozy_residues} that the set of quadratic residues $R$ is primitive
and proved that  if $R=A+B$, $|A|,|B|\ge 2$ then cardinalities of $A,B$ should be close to $\sqrt{p}$.
There are several papers in the direction, see \cite{BMR}, \cite{DS_AD}, \cite{LS}, \cite{Sh_Sarkozy}, \cite{SSV}, \cite{Shparlinski_AD}.
%His
S\'{a}rk\"{o}zy's
result was refined slightly in \cite{Shparlinski_AD}, \cite{Sh_Sarkozy}
and extended to the case of all multiplicative subgroups by Shparlinski, see \cite{Shparlinski_AD}.

\begin{theorem}
    Let $\G \subseteq \f_p$ be a multiplicative subgroup and for some $A,B \subseteq \f_p$
    one has
\begin{equation}\label{f:Shparlinskii-}
    A+B \subseteq \G \,,
\end{equation}
    where $|A|,|B| \ge 2$.
    Then
\begin{equation}\label{f:Shparlinskii}
    |A|,\, |B| \le |\G|^{1/2 + o(1)}
\end{equation}
    as $|\G| \to \infty$.
    In particular, if $A+B = \G$ then
$$
    |A|,\, |B| = |\G|^{1/2 + o(1)} \,.
$$
\label{t:Shparlinskii}
\end{theorem}

Note that if it is known that $A$ and $B$ have comparable sizes than one do not need in the restriction
$|A|,|B| \ge 2$ to get (\ref{f:Shparlinskii}).
Also the inclusion (\ref{f:Shparlinskii-}) can be replaced by a little bit wider one, namely,  $A+B\subseteq \G \bigsqcup \{ 0 \}$,
see the proof from \cite{Shparlinski_AD}.

%In the paper we prove that all sufficiently small multiplicative subgroups are primitive.
%\begin{theorem}
%    Let $\G \subseteq \f_p$ be a multiplicative subgroup, $\eps>0$ be any real and $|\G| \le p^{2/3-\eps}$.
%    Then $\G$ is primitive, that is $\G \neq A+B$ for all $A,B\subseteq \f_p$ with $|A|, |B| \ge 2$.
%\label{t:AD_subgroups}
%\end{theorem}

\bigskip

In the paper we consider a particular case when $A$ or $B$ is contained in some shift of a multiplicative subgroup.
It turns out that the exponent $1/2$ from (\ref{f:Shparlinskii}) can be refined in the situation.
Let us formulate our result in the simplest symmetric
%situation.
case (that is $B=-A$).
Denote by $\f^*_p$ the set $\f_p \setminus \{ 0 \}$.

\begin{theorem}
    Let $A,\G \subseteq \f_p$ be multiplicative subgroups.
    Suppose that for some $\xi \in \f^*_p$ one has
    \begin{equation}\label{cond:A-A_total}
        A-A \subseteq \xi \G \bigsqcup \{ 0\} \,.
    \end{equation}
    If $|\G| < p^{3/4}$ then
    $$
        |A| \ll |\G|^{1/3+o(1)} \,.
    $$
    If $|\G| \ge p^{3/4}$ then
    $$
        |A| \ll \min\left\{ \frac{|\G| \log^{1/2} p}{\sqrt{p}}, \sqrt{p} \right\} \,.
    $$
\label{t:A-A_total}
\end{theorem}

In particular,  Theorem \ref{t:A-A_total} implies that   for any
    %$\kappa < 5/6$, and $|\G| \le p^{\kappa}$
    $\G$ with
    %$|\G| \ll p^{}/ \log p$
    $|\G| \ll p^{1-o(1)}$
    one has $\G \bigsqcup \{ 0\} \neq A-A$, where
    %$|A| \ge 2$.
    size of
    %$A$
    $\G$
    is sufficiently large.
    What can be said in the non--symmetric situation?
    In \cite{SSV} the following result on the decompositions was proved.

\begin{proposition}\label{p:subgroup_decomp_impossible}
Let $\eps \in (0,1]$ be a real number,
$A, \G \subset \mathbb{F}_p^{*}$ be sufficiently large multiplicative subgroups,
and $B \subseteq \mathbb{F}_p$ be an arbitrary nonempty set.
If $|\G \cap A| \ll |A|^{1-\eps}$ and $|\G| \ll p^{1-\eps/6},$
then $\G$ has no nontrivial representations as $\G=A+B.$
\end{proposition}

Now we can drop additional assumptions on the intersection $A$ and $\G$.

\begin{theorem}\label{t:subgroup_decomp_impossible_new}
Let $\eps \in (0,1]$ be a real number.
Let also $A, \G \subset \mathbb{F}_p^{*}$ be two sufficiently large multiplicative subgroups,
$|\G| \le p^{1/2-\eps}$,
and $B \subseteq \mathbb{F}_p$ be an arbitrary nonempty set.
Then $\G$, $\G \bigsqcup \{ 0 \}$ has no nontrivial representations as $\G=A+B.$
\end{theorem}

In the proof of Theorem \ref{t:subgroup_decomp_impossible_new} we refine condition (\ref{cond:A-A_total})
and replace it by a more flexible, namely,
\begin{equation}\label{cond:xi_eta_intr}
    \xi A + \eta A \subseteq \G \bigsqcup \{ 0 \} \,,
\end{equation}
where $\xi, \eta \in \f^*_p$ are arbitrary.
Moreover, one can consider large subsets of $\xi A$, $\eta A$ satisfying inclusion (\ref{cond:xi_eta_intr}) and even take different subgroups $A_1$, $A_2$ of comparable sizes, see Theorem \ref{t:sumsets_xi_eta} and Corollary \ref{cor:A+A} of section \ref{sec:proof}.
Besides one can deal with the intersections of the form $(\xi A + \eta A) \cap \G$ instead of inclusion
(\ref{cond:xi_eta_intr}), see Proposition \ref{p:sumsets_intersections} of section \ref{sec:CR}.

In \cite{LS} the following problem was considered.
Let as above $R \subset \f^*_p$ be the subgroup of quadratic residues.
%, that is the set of nonzero squares.
Can it be
%that
$R \doteq A- A$,
%, in other words
that is
any $x\in R$ is represented as a difference of two elements of $A$ in a unique way and $A-A \subseteq R \bigsqcup \{ 0 \}$?
The authors of \cite{LS} proved, in particular, that
%Then
for any $\xi \in \f^*_p$ one cannot has
$R \doteq \xi A-\xi A$, where $A$ is a multiplicative subgroup.
In section \ref{sec:CR}, we obtain a generalization, see Corollary \ref{c:CR} as well as remarks below the corollary.
% and Proposition \ref{p:CR}.

\begin{theorem}
    Let $A, \G \subset \mathbb{F}_p^{*}$ be two
    multiplicative subgroups.
    %, $|A| >1$.
    Suppose that $A-A = \xi \G \sqcup \{ 0 \}$ for some $\xi \in \mathbb{F}_p^{*}$
    and $(A\c A) (x) = c$ is a constant onto $\G$.
    Then $|A|^2-|A| = c|\G|$ and either $|\G| = O(1)$ or $|\G| \gg \frac{p}{\log p}$.
\label{t:LS_my_intr}
\end{theorem}

Let us say a few words
%on
about the
methods of the proving of Theorem \ref{t:Shparlinskii} and similar results
%and
%as well as
and compare them  with
our new approach.
Suppose that for some set $S$, we have $S=A+B$.
The main observation of papers \cite{Sarkozy_residues}, \cite{DS_AD}, \cite{Shparlinski_AD}, see also \cite{Sh_energy}
is the following.
Take an
arbitrary positive integer $k$ and any elements $b_1,\dots,b_k\in B$.
Then by the definition of sumset, we have
\begin{equation}\label{f:A,D_inclusion}
    A \subseteq (S - b_1) \bigcap (S - b_2) \dots \bigcap (S-b_k)
    \,.
\end{equation}
Thus, if $S$ is a sumset then there are large intersections of $S$ with its shifts of form (\ref{f:A,D_inclusion}).
On the other hand, in the case of large multiplicative subgroups one can
use
%apply
analytical tools (see e.g. Lemma \ref{l:C_for_subgroups} below) to show that for
%any
$x_1,\dots,x_k \in \f_p$ there is a uniform upper bound
$$|(S + x_1) \bigcap (S + x_2) \dots \bigcap (S+x_k)| \ll p^{1/2 + o(1)} \,,$$
provided by $x_1,\dots,x_k$ are nonzero and pairwise distinct elements.
For smaller subgroups the
%following
analogous
result from \cite{V-S}
%\cite{SSV}
makes the job.

\begin{theorem}
    Let $\G\subseteq \f_p$ be a multiplicative subgroup,
    $k\ge 1$ be a positive integer, and $x_1,\dots,x_k$ be different nonzero elements.
    Let also
    $$
        32 k 2^{20k \log (k+1)} \le |\G|\,, \quad  p \ge 4k |\G|  ( |\G|^{\frac{1}{2k+1}} + 1 ) \,.
    $$
    Then
    \begin{equation}\label{f:main_many_shifts}
        |\G\bigcap (\G+x_1) \bigcap \dots (\G+x_k)| \le 4 (k+1) (|\G|^{\frac{1}{2k+1}} + 1)^{k+1} \,.
    \end{equation}
    The same holds if one replace $\G$ in (\ref{f:main_many_shifts}) by any cosets of $\G$.
\label{t:main_many_shifts}
\end{theorem}

Thus,
%the result
theorem above
asserts us that
$|\G\bigcap (\G+x_1) \bigcap \dots (\G+x_k)| \ll_k |\G|^{\frac{1}{2}+\alpha_k}$,
%%$|R\bigcap (R+\mu_1) \bigcap \dots \bigcap (R+\mu_k)| \ll_k |R|^{\frac{k+1}{2k+1}}$,
provided by
$1 \ll_k |\G| \ll_k p^{1-\beta_k}$,
%$1 \ll_k |R| \ll_k p^{2/3-\alpha_k}$,
where $\alpha_k, \beta_k$ are some sequences of positive numbers, and $\alpha_k, \beta_k \to 0$, $k\to \infty$.
A little bit better
bounds
%bounds for the sequences
%dependence on parameters
can be found
%are contained
in \cite{SSV}.

So, using inclusion (\ref{f:A,D_inclusion}), combining with the analytical tools or Theorem \ref{t:main_many_shifts},
we obtain Theorem \ref{t:Shparlinskii}.
Now, the main question is: can we prove primitivity of multiplicative subgroups in this way, i.e.
just
%using
combining
upper
bounds on intersections
%above?
of the same strength as
in (\ref{f:main_many_shifts}) and some random properties of $\G$ as smallness of its Fourier coefficients,
%say?
e.g.?
As was pointed in \cite{Sh_Sarkozy}, \cite{Sh_energy} the answer is no.
To see this consider so--called "random sumsets"\, example from
%, see
\cite{Sh_energy}.
Let $A\subseteq \f_p$ be a random set of size $o(\sqrt{p})$,
%such that $|A+A| = o(p)$,
say.
The set $A+A$ is what we call a random sumset.
It easy to see that $A+A$ has rather random behaviour, excepting nonrandom property (\ref{f:A,D_inclusion}), of course.
Thus, to prove that a set is primitive we cannot use just random properties of the set or upper bounds for the intersections as in (\ref{f:main_many_shifts}).
We need in different tools.
For example,
the first author observed
in \cite{Sh_Sarkozy}
that quadratic residues $R$ is an (almost) perfect difference
%applied the arguments
set, that is the function $f(x) = |\{ r_1-r_2 ~:~ r_1,r_2 \in R \}$ is (almost) constant on $\f_p \setminus \{ 0 \}$.
Certainly, random sets have no such a property.
This allowed him to prove that $R$ cannot be represented in the form $A+A$ and even to be close to
%the set
$A+A$
in some sense.

In the article we use another
nonrandom
property  of multiplicative subgroups.
Suppose for simplicity that we are in a symmetric situation, that is $A-A \subseteq \G \bigsqcup \{ 0 \}$
and $|A| \sim |\G|^{1/2 + o(1)}$.
One can assume that $A^*\subseteq \G$, where $A^* = A\setminus \{0\}$ and $1\in A^*$.
Then it is easy to see that $A^*/A^* \subseteq \G \cap (\G+1)$.
If $|\G| < p^{3/4}$ then in view of Theorem \ref{t:main_many_shifts} it gives us $|A^*/A^*| \le |\G \cap (\G+1)| \ll |\G|^{2/3}$
and the bound implies a non--trivial lower bound for the multiplicative energy of
$A$, namely,
$\E^{\times} (A) \ge |A|^{8/3+o(1)}$, see the definition in section \ref{sec:definitions}.
Certainly, a random set $A$ has no such a property and hence we have separated from the random sumset case.
Unfortunately, the lower bound for the multiplicative energy is not enough to prove
%Theorems \ref{t:AD_subgroups}, \ref{t:A-A}
our main results
because of weakness of estimate (\ref{f:main_many_shifts}).
In the proof we use a stronger lower bound for average value of common energies
$\T (A) := \sum_{a_1,a_2 \in A} \E^\times (A-a_1,A-a_2)$, see section \ref{sec:fr}.
This quantity appears in papers \cite{Jones_PhD}, \cite{R_Minkovski}, \cite{Sh_3G} and others.
It was showed in  \cite{Sh_3G} that, in contrary,
%such quantity
$\T (A)$
is small for any multiplicative subgroup $A$
and we arrive to a contradiction.
Thus, in principle, the methods of the paper work for arbitrary sets $A$ with small $\T(A)$.
For example, they work
when $A$ is an arithmetic progression,
for the contrary to the subgroup case
(%in contrary,
on the other hand,
we are crucially need in the fact that the container set $\G$ is a subgroup).
%fact that  $\la A/A \subseteq \G \cap (\G+\la)$ for any $\la \in \G$.
Finally,
%we need in the upper bound for size of subgroup
we have deal with subgroups of small size only
to separate from the case of the largest subgroup, namely,
$\f^*_p$ which is, clearly, additively decomposable in many ways.
%A stronger assumption $|\G| \le p^{2/3-\eps}$

The paper is organized as follows.
In section \ref{sec:definitions} we explain  some of notation that will be used later
and give
%basic results.
auxiliary results on matrices.
Section \ref{sec:operators} is also devoted to a notation required
for
the operator technique from \cite{SS1,s_mixed}.
In the next section \ref{sec:fr} we prove Theorem \ref{t:A-A_total}.
The methods here are elementary and do not require eigenvalues
%method
approach
from \cite{SS1,s_mixed} as well as the
notation from
%section \ref{sec:operators}.
the previous section.
Unfortunately, these elementary observations work just in the difference case  (\ref{cond:A-A_total}).
The general situation is considered in section \ref{sec:proof}.
In particular, here we
%obtain
get
Theorem \ref{t:subgroup_decomp_impossible_new}.
The method of the proving allows us to obtain Theorem \ref{t:LS_my_intr} in next section \ref{sec:CR}.
%and also we describe  all exceptional subgroups
%with
%%having the property
%$\G \doteq A-A$.
% here.
In the last section \ref{sec:T_A_f_p} we discuss further properties of the quantity $\T(A)$
and prove, in particular, a "dual"\, (that is replacing the addition by the multiplication  and vice versa) sum--products estimate, concerning the
sum
%quantity
$\sum_{c\in C} \E^\times (A-c,B)$, see Proposition \ref{p:semi_T}
below.

The author
%I.S.
is grateful to Dmitry Zhelezov, Elena Solodkova
%and
%Elena Solodkova
for useful discussions
and
%as well as
Peter Hegarty as well as G\"{o}teborg University and Chalmers University for their hospitality.
%and, especially, Tomasz Schoen for very useful and fruitful explanations and  discussions.

\section{Notation and auxiliary results}
\label{sec:definitions}

%We conclude with few comments regarding the notation used in this paper.
%Finally, with a slight abuse of notation
In the paper we use the same letter to denote a set $S\subseteq \f_p$
and its characteristic function $S:\f_p \rightarrow \{0,1\}.$
By $|S|$ denote the cardinality of $S$.

Let $f,g : \f_p \to \C$ be two functions.
Put
\begin{equation}\label{f:convolutions}
    (f*g) (x) := \sum_{y\in \f_p} f(y) g(x-y) \quad \mbox{ and } \quad
        %(f\circ g) (x) := \sum_{y\in \Gr} \ov{f(y)} g(y+x)
        (f\circ g) (x) := \sum_{y\in \f_p} f(y) g(y+x)
\end{equation}
If $\gamma \in \f_p^*$ then $f^\gamma (x) := f(\gamma x)$.
%Write
Put
$\E^{+}(A,B)$ for the {\it additive energy} of two sets $A,B \subseteq \f_p$
(see e.g. \cite{TV}), that is
$$
    \E^{+} (A,B) = |\{ a_1+b_1 = a_2+b_2 ~:~ a_1,a_2 \in A,\, b_1,b_2 \in B \}| \,.
$$
If $A=B$ we simply write $\E^{+} (A)$ instead of $\E^{+} (A,A).$
Clearly,
\begin{equation*}\label{f:energy_convolution}
    \E^{+} (A,B) = \sum_x (A*B) (x)^2 = \sum_x (A \circ B) (x)^2 = \sum_x (A \circ A) (x) (B \circ B) (x)
    \,.
\end{equation*}
Note also that
\begin{equation}\label{f:E_CS}
    \E^{+} (A,B) \le \min \{ |A|^2 |B|, |B|^2 |A|, |A|^{3/2} |B|^{3/2} \} \,.
\end{equation}
In the same way define the {\it multiplicative energy} of two sets $A,B \subseteq \f_p$
$$
    \E^{\times} (A,B) = |\{ a_1 b_1 = a_2 b_2 ~:~ a_1,a_2 \in A,\, b_1,b_2 \in B \}| \,.
$$
Certainly, multiplicative energy $\E^{\times} (A,B)$ can be expressed in terms of multiplicative convolutions,
similar to (\ref{f:convolutions}).

Denote by
$$ \Cf_{k+1} (f_1,\dots,f_{k+1}) (x_1,\dots, x_{k})$$
the function
$$
    %\Cf_k(F) (x) =
    \Cf_{k+1} (f_1,\dots,f_{k+1}) (x_1,\dots, x_{k}) = \sum_z f_1 (z) f_2 (z+x_1) \dots f_{k+1} (z+x_{k}) \,.
$$
Thus, $\Cf_2 (f_1,f_2) (x) = (f_1 \circ f_2) (x)$.
%Put $C_1 (f) = \sum_z f(z)$.
If $f_1=\dots=f_{k+1}=f$ then write
$\Cf_{k+1} (f) (x_1,\dots, x_{k})$ for $\Cf_{k+1} (f_1,\dots,f_{k+1}) (x_1,\dots, x_{k})$.

\bigskip

We conclude by two auxiliary results.
The first one is from the ordinary theory of matrix inequalities, see e.g. \cite{H-J}.

\begin{lemma}
    Let $M$ be a normal $(n\times n)$--matrix with eigenvalues $\mu_1,\dots,\mu_n$,
    and let $f$ be an arbitrary convex function of $n$ real variables.
    Then
$$
    \max_{ x_1, \dots, x_n } f( \langle M x_1, x_1 \rangle, \dots, \langle M x_n, x_n \rangle)
        =
            \max_{i_1,\dots, i_n} f(\mu_{i_1},\dots,\mu_{i_n}) \,,
$$
where the maximum on the left-hand side is taken over all orthonormalized systems of
vectors $x_1,\dots,x_n$, and the right-hand maximum over an arbitrary permutation of the
numbers $\{1,2,\dots,n\}$.
\label{l:convex_eigenvalues}
\end{lemma}

The second
%one
lemma
is also rather standard, see Theorem 2.5.4 of \cite{H-J}.

\begin{lemma}
     Let $M = (m_{ij})$ be any $(n\times n)$--matrix with eigenvalues $\mu_1,\dots,\mu_n$.
     Then
$$
    \sum_{j=1}^n |\mu_j|^2 \le \sum_{i,j=1}^n  |m_{ij}|^2 \,,
$$
    and the equality iff $M$ is a normal matrix.
\label{l:sum_of_squares}
\end{lemma}

%Also, we need in a general result from \cite{R-N_R_S}.

%\begin{theorem}
%\label{t:sum-prod}
%    Let $A,B,C\subseteq F$, let $M = \max(|A|,|BC|).$ Suppose that $|A||B||BC| \ll p^2$. Then
%\begin{equation}\label{f:sum-prod_energy}
%    \E^{+} (A,C) \ll (|A||BC|)^{3/2} |B|^{-1/2} + M |A||BC| |B|^{-1} \,.
%\end{equation}
%\end{theorem}

We finish the section by two results from additive combinatorics.
The first one is the famous Balog--Szemer\'{e}di--Gowers theorem.

\begin{theorem}
    Let $A,B \subseteq \Gr$ be two sets such that $\E^{+} (A,B) \ge |A|^{3/2} |B|^{3/2} / K$.
    Then there are $x,y\in \Gr$ and a symmetric set $H\subseteq \Gr$ with
    $|A \cap (H+x)|, |B \cap (H+y)| \gg K^{-M} |H|$, further,
    $|A|, |B| \ll K^{M} |H|$ and $|H+H| \ll K^{M} |H|$.
    Here $M>0$ is an absolute constant.
\label{t:BSzG_A-B}
\end{theorem}

The second result is due to D. Zhelezov \cite{Z_A(A+1)}.

\begin{theorem}
    Let $A,B,C \subset \F_p$ be three sets, $|A|=|B|=|C| \le \sqrt{p}$.
    Then for any fixed $d\neq 0$ holds
$$
    \max\{ |AB|, |(A+d)C|\}\gg |A|^{1+1/26} \,.
$$
\label{t:Z}
\end{theorem}

All logarithms are base $2.$ Signs $\ll$ and $\gg$ are the usual Vinogradov's symbols.
If we have a set $A$ then we will write $a \lesssim b$ or $b \gtrsim a$ if $a = O(b \cdot \log^c |A|)$, $c>0$.

\section{On operators over multiplicative subgroups}
\label{sec:operators}

Let $\Gr$ be an abelian group.
Let also $g : \Gr \to \C$ be a function, and $A\subseteq \Gr$ be a finite set.
By $\oT^{g}_A$ denote the matrix with indices in the set $A$
\begin{equation}\label{def:operator1}
    \oT^{g}_A (x,y) = g(x-y) A(x) A(y) \,.
\end{equation}
It is easy to see that $\oT^{g}_A$ is hermitian iff $\ov{g(-x)} = g(x)$.
The corresponding action of $\oT_A^g$ is
$$
    \langle \oT^{g}_A a, b \rangle = \sum_z g(z) (\ov{b} \c a) (z) \,.
$$
for any functions $a,b : A \to \C$.
In the case $\ov{g(-x)} = g(x)$ by $\Spec (\oT^{g}_A)$ we denote the spectrum of the operator $\oT^{g}_A$
$$
    \Spec (\oT^{g}_A) = \{ \mu_1 \ge \mu_2 \ge \dots \ge \mu_{|A|} \} \,.
$$
Write  $\{ f \}_{\a}$, $\a\in [|A|]$ for
the corresponding eigenfunctions.
Let us note two simple formulas
\begin{equation}\label{f:trace1}
    \sum_{\a=1}^{|A|} \mu_\a = g(0) |A| \,,
\end{equation}
and
\begin{equation}\label{f:trace2}
    \sum_{\a=1}^{|A|} |\mu_\a|^2 = \sum_{x} |g(x)|^2 (A\c A) (x) \,.
\end{equation}
General theory of such operators was developed in \cite{SS1,s_mixed}.

\bigskip

Now we consider the operators of a special form.
Let $p$ be a prime number, $q=p^s$ for some integer $s \ge 1$.
Let
$\F_q$ be the field with  $q$ elements, and let $\Gamma\subseteq \F_q$ be a multiplicative subgroup.
We will write $\F^*_q$ for $\F_q\setminus \{ 0 \}$.
Denote by $t$ the cardinality of $\Gamma$, and put $n=(q-1)/t$. Let
also $\mathbf{g}$ be a primitive root, then $\Gamma = \{ \mathbf{g}^{nl} \}_{l=0,1,\dots,t-1}$.
Let $\{ \chi_\a (x)\}_{\a \in [t]}$ be the
orthogonal
family of multiplicative characters on $\Gamma$ and
$\{ f_\a (x) \}_{\a \in [t]}$ be the correspondent orthonormal family, that is
\begin{equation}\label{f:chi_Gamma}
    f_\a (x) = |\G|^{-1/2} \chi_\a (x) = |\G|^{-1/2} \cdot e\left( \frac{\a l}{t} \right) \,, \quad x=\mathbf{g}^{nl} \,, \quad 0\le l < t \,.
\end{equation}
In particular, $f_\a (x) = \chi_\a (x) = 0$ if $x\notin \G$.
Clearly, products of such functions form a basis on Cartesian products of $\G$.

If $\_phi:\G \to \C$ be a function then denote by $c_\a (\_phi)$ the correspondent coefficients of $\_phi$ relatively to the family $\{ f_\a (x) \}_{\a \in [t]}$. In other words,
$$
    c_\a (\_phi) := \langle \_phi, f_\a \rangle = \sum_{x \in \G} \_phi (x) \ov{f_\a (x)} \,, \quad \quad \a \in [|\G|] \,.
$$

The method of the paper based on
%We need in
the lemma, which was proved mainly in \cite{SS1}.
We give the proof for the sake of completeness.
Further results on the spectrum of operators connected with multiplicative subgroups can be found in \cite{Sh_energy}.

\begin{lemma}
    Let $\Gamma\subseteq \F^*_q$ be a multiplicative subgroup.
    Suppose that $H(x,y) : \G \times \G \to \C$ satisfies two conditions
\begin{equation}\label{f:subgroup_eigenvalues}
    H(y,x) = \ov{H(x,y)} \quad \mbox{ and } \quad H(\gamma x, \gamma y) = H(x,y)\,, \quad \forall \gamma \in \G \,.
\end{equation}
    Then the functions $\{ f_\a (x) \}_{\a \in [|\G|]}$ form the complete orthonormal family of the eigenfunctions of the operator $H(x,y)$.
\label{l:subgroup_eigenvalues}
\end{lemma}
\begin{proof}
    The first property of (\ref{f:subgroup_eigenvalues}) says that $H$ is a hermitian operator, so it has a complete orthonormal family of its eigenfunctions.
    Consider the equation
\begin{equation}\label{tmp:25.04.2015_2}
    \mu f(x) = \G(x) \sum_{y\in \G} H(x,y) f(y) \,,
\end{equation}
    where $\mu$ is some number and $f : \G \to \C$ is unknown function.
    It is sufficient to check that any $f=\chi_\a$, $\a \in [|\G|]$ satisfies the equation above.
    Indeed, making a substitution $x\to x\gamma$ into (\ref{tmp:25.04.2015_2}) and using the characters property, we obtain
$$
    \mu f(x) f(\gamma) = \G(x\gamma) \sum_{y} H(\gamma x,y) f(y) = \G(x) \sum_{y} H(\gamma x,\gamma y) f(\gamma y)
        =
            \G(x) f(\gamma) \sum_{y} H(x,y) f(y) \,,
$$
    where the second property of (\ref{f:subgroup_eigenvalues}) has been used.
    Thus, it remains to check (\ref{tmp:25.04.2015_2}) just for one $x\in \G$.
    Choosing the number $\mu$ in an appropriate way we attain  the former.
    This completes the proof.
$\hfill\Box$
\end{proof}

\begin{corollary}
    Let $\Gamma\subseteq \F^*_q$ be a multiplicative subgroup
    and $g$ be any $\G$--invariant real function.
    Then the operator $\oT^g_\G$ is normal.
\end{corollary}

%\section{The proof}
\section{First results}
\label{sec:fr}

In the section we have deal with the quantity $\T (A,B,C,D)$, see \cite{R_Minkovski}, \cite{Sh_3G}
($\T$ for collinear {\it triples})
\begin{equation}\label{f:T_energy}
    \T (A,B,C,D) := \sum_{c \in C,\, d\in D} \E^\times (A-c,B-d) \,.
\end{equation}
Clearly, $\T (A,B,C,D)$ enjoy the following invariance property
\begin{equation}\label{f:T_invariance}
    \T (A-x,B-y,C-x,D-y) = \T (A,B,C,D) \,,\quad \quad \forall x,y
    %\,.
\end{equation}
as well as
\begin{equation}\label{f:T_invariance_m}
    \T (\la A, \mu B, \la C, \mu D) = \T (A,B,C,D) \,,\quad \quad \forall \la, \mu \neq 0
    \,.
\end{equation}
If $A=B$, $C=D$ then denote by $\T(A,C)$ the quantity $\T (A,A,C,C)$.
If $A=B=C=D$ then we write $\T(A)$ for $\T (A,A,A,A)$.
Let us make a few remarks about the quantity $\T (A,B,C,D)$.
%Using the Cauchy--Schwarz inequality one can obtain
%\begin{equation}\label{f:T_CS}
%    \T^2 (A,B,C,D) \le \T(A,C) \T(B,D) \,.
%\end{equation}
It is easy to check that
$$
    \T (A,B,C,D) = \sum_{x,x' \neq 0} \sum_{\la} \Cf_3 (C,A,A) (x,\la x) \cdot \Cf_3 (D,B,B) (x',\la x')
        +
$$
\begin{equation}\label{f:T_error}
        +
            \theta (|A\cap C| |B|^2 |D| + |B\cap D| |A|^2 |C| + 2 |A\cap C| |B\cap D| |A| |B|) \,,
\end{equation}
where $|\theta| \le 1$.
Three error terms in (\ref{f:T_error}) are usually negligible.
Denote by $\T^* (A,B,C,D)$ the rest.
Thus, in the symmetric case $A=B$, $C=D$ one has
\begin{equation}\label{f:T_square}
    \T^* (A,C) = \sum_\la \left( \sum_{x\neq 0} \Cf_3 (C,A,A) (x,\la x) \right)^2 \,.
\end{equation}
Finally, because $\E^\times (A-c,B-b) \ge |A| |B|$ it follows that $\T (A,B,C,D) \ge |A||B||C||D|$.
It turns out that there is the same upper bound for $\T (A)$ (up to logarithmic factors)
in the case of cosets of a multiplicative subgroup $A$.
The proof based on the following lemma of Mit'kin~\cite{Mitkin}, see also \cite{K_Tula}, \cite{SSV} and \cite{V-S}.
%%A similar result was obtained by Mit'kin~\cite{Mitkin}.

\begin{lemma}
\label{l:mitkin}
Let $p>2$ be a prime number, $\Gamma,\Pi$ be subgroups of $\mathbb{F}_p^*,$
$M_\Gamma,M_\Pi$ be sets of distinct coset representatives of $\Gamma$ and $\Pi$, respectively.
For an arbitrary set $\Theta \subset M_\Gamma \times M_\Pi$ such that $(|\Gamma||\Pi|)^2|\Theta| < p^3$
and $|\Theta| \le 33^{-3}|\Gamma||\Pi|$, we have
\begin{equation}\label{eq:mitkin}
\sum_{(u,v) \in \Theta}\Bigl|\{(x,y) \in \Gamma \times \Pi : ux+vy=1\}\Bigr| \ll (|\Gamma||\Pi||\Theta|^2)^{1/3}.
\end{equation}
\end{lemma}

Using the above lemma, the required upper bound for $\T$ was obtained in \cite{Sh_3G}.
It is easy to establish a similar result for larger subgroups $A$, see Proposition \ref{p:sigma_l} of section \ref{sec:T_A_f_p}.

\begin{proposition}
Let $p$ be a prime number, $\Gamma,\Pi$ be subgroups of $\mathbb{F}_p^*$.
Suppose that $|\Gamma||\Pi| < p$.
Then for any $\xi \in \f_p^*$, $\eta \in \f_p^*$ one has
\begin{equation}\label{f:sigma}
%    \sum_{\gamma \in \G,\, \pi\in \Pi} \E^\times (\G-\gamma,\Pi-\pi)
    \T (\G, \Pi, \xi \G, \eta \Pi)
        \ll
            |\G|^2 |\Pi|^2 \log (\min\{ |\G|, |\Pi| \}) + |\G| |\Pi| (|\G|^2 + |\Pi|^2) \,.
\end{equation}
\label{p:sigma}
\end{proposition}

\bigskip

It is known that any sets $A$, $B$ with  $A+B\subseteq R$ ($R$ is the set of all quadratic residues) satisfy $|A||B| < p$, see e.g. \cite{BMR}.
Before proving the first  main result of the section let us note a very simple estimate for the sizes of $A,B$ with $A+B\subseteq \G \bigsqcup \{ 0 \}$, where $\G$ is a multiplicative subgroup.

\begin{lemma}
    Let $\G \subset \mathbb{F}_p^{*}$ be a multiplicative subgroup,
    and $A,B\subseteq \f_p$ be two sets.
    Suppose that $A+B\subseteq \G \bigsqcup \{ 0 \}$.
    Then $|A| |B| < 4 p$.
\label{l:ab_p}
\end{lemma}
\begin{proof}
    Suppose that $|A||B|\ge 4p$.
    By the assumption $\G \neq \f_p^*$.
    It follows that $|\G| \le (p-1)/2$.
    Clearly, we can assume that $|\G|+1 \ge 2\sqrt{p}$.
    On the other hand by the Cauchy--Davenport  theorem \cite{TV}, we get
$$
    \frac{p-1}{2} \ge |\G| \ge |A|+|B|-2 \ge 4\sqrt{p} -2 \,.
$$
    It follows that $p\ge 59$.
    We have for any multiplicative subgroup $\G$
    %that
    the following upper bound for its Fourier coefficients
$$
    \max_{\xi \neq 0} |\sum_{x\in \G} e^{2\pi i x \xi /p} | \le \sqrt{p-|\G|} \,,
$$
    e.g. see \cite{KS1}.
    Thus by the conditions $A+B\subseteq \G \bigsqcup \{ 0 \}$, $\G \neq \f_p^*$ and simple Fourier analysis, we
    %obtain
    get
$$
    |A| |B| p < |A| |B| (|\G|+1) + (\sqrt{p-|\G|}+1) (|A| |B|)^{1/2} ((p-|A|) (p-|B|))^{1/2} \,.
$$
    After
    %simple
    some
    calculations, we obtain $|A| |B| < 4 p$ as required.
$\hfill\Box$
\end{proof}

%\bigskip

%\begin{corollary}
%    Let $A, \G \subset \mathbb{F}_p^{*}$ be multiplicative subgroups,
%    Suppose that $A-A \subseteq \xi \G \sqcup \{ 0 \}$ for some $\xi \in \mathbb{F}_p^{*}$.
%    Then $\T(A) \ll |A|^4 \log |A|$.
%\label{cor:strange}
%\end{corollary}
%\begin{proof}
%    By Lemma \ref{l:ab_p} we see that $|A| < 2 \sqrt{p}$.
%    If $A$ has nonempty intersection with quadratic residues and non--residues then
%    the sizes of the intersections are less than $\sqrt{p}$ and we
%$\hfill\Box$
%\end{proof}

\bigskip

Now we can prove the first main result of the section.

\begin{theorem}
    Let $A, \G \subset \mathbb{F}_p^{*}$ be two
    %sufficiently large
    multiplicative subgroups,
    $|A|<\sqrt{p}$ and
    $|\G| < p^{3/4}$.
    Suppose that $A' \subseteq A+s$, $s\in \f_p$ is an arbitrary, and $A'-A' \subseteq \xi \G \sqcup \{ 0 \}$ for some $\xi \in \mathbb{F}_p^{*}$.
    Then
$$
    |A'|^6 \ll |A|^4 |\G|^{2/3} \log^{} |A| \,.
$$
    In particular, if $A-A \subseteq \xi \G \sqcup \{ 0 \}$ then
    $
        |A| \ll |\G|^{1/3} \log^{1/2} |\G| \,.
    $
\label{t:A-A_in_G}
\end{theorem}
\begin{proof}
    In the case $A'=A$ one can assume that $|A| < \sqrt{p}$ for sufficiently large $\G$, see
    Theorem \ref{t:Shparlinskii}.
    In general case, applying for\-mu\-la (\ref{f:T_square}) with $A=B=C=A'$ and using the Cauchy--Schwarz inequality, we get
$$
    |A'|^6 \ll (|A'|^3 - 2|A|^2)^2 \le \left( \sum_{x \neq 0,\, \la \neq 0} \Cf_3 (A') (x, \la x) \right)^2
        \ll
$$
\begin{equation}\label{tmp:01.04.2015_2}
        \ll
            \T(A+s) \cdot |\{ \la \neq 0 ~:~ \exists x \neq 0 \mbox{ ~s.t.~ } \Cf_3 (A,A,A) (x,\la x) \neq 0 \}| \,.
\end{equation}
    Clearly, if there is $x\neq 0$ with $\Cf_3 (A,A,A) (x,\la x) \neq 0$ then
    $\la = (a_1-a)/(a_2-a) \neq 0$
    for some $a_1,a_2,a\in A$.
    Since  $A-A \subseteq \xi \G \sqcup  \{ 0 \}$ it follows that $\la \in \G$.
    But $\la-1 = (a_1-a_2)/(a_2-a) \in \G \sqcup \{ 0 \}$.
    Hence $\la \in \G\cap ((\G \sqcup \{ 0 \})+1)$.
    We have $|\G| < p^{3/4}$.
    By Lemma \ref{l:mitkin}
    %we have
    it follows that
    $|\G\cap ((\G \sqcup \{ 0 \})+1)| \ll |\G|^{2/3}$.
    Note, finally, that by (\ref{f:T_invariance}), we have $\T(A+s) = \T(A)$.
    Returning to (\ref{tmp:01.04.2015_2}) and using Proposition \ref{p:sigma}, we obtain
$$
    |A'|^6 \ll |A|^4 \log |A| \cdot |\G|^{2/3}
$$
    as required.
$\hfill\Box$
\end{proof}

\bigskip

\begin{remark}
    Using a result from \cite{J_R-N} on lower bound for the size of the set of the form $C(C+1)$,
    combining with the arguments of \cite{Bourgain_more}, one can prove that
    for any $B\subset \f_p$, $|B| < \sqrt{p}$, one has $|B\cap (B \pm 1)| \lesssim |BB|^{54/55}$.
    Thus, the arguments of the proof of Theorem \ref{t:A-A_in_G} gives a non--trivial upper bound for {\bf any} set $A$ with $A-A \subseteq B$, namely, $|A|^6 \lesssim \T(A) |BB/BB|^{54/55}$.
    Of course, there are another
    %variants
    ways
    to prove the same, applying lower bounds for the cardinality of
    $(A\pm A)/(A\pm A)$, say.
    In the case when $A$ is a multiplicative subgroup it is more effectively to use a bound from \cite{Sh_3G},
    namely $|\{\frac{a_1-a}{a_2-a} ~:~ a,a_1,a_2\in A\}| \gg \frac{|A|^2}{\log |A|}$, $|A| < \sqrt{p}$.
    Using the Pl\"{u}nnecke's inequality \cite{TV}, it gives us $|A| \ll |BB| |B|^{-1/2} \log^{1/2} |B|$
    and $|A| \lesssim  |BB|^{108/55} |B|^{-81/55} \log^{1/2} |B|$.
\end{remark}

\bigskip

\begin{remark}
    As was pointed in \cite{LS}
    %that
    if $A,\G$ are two multiplicative subgroups and $A - A \subseteq \xi \G \sqcup \{ 0 \}$
    then because of
    $a^2_1 - a^2_2  = (a_1 - a_2) (a_1+a_2)$ we have $\{ a_1 + a_2 ~:~ a_1 \neq a_2,\, a_1,a_2 \in A \} \subseteq \G$.
    Thus the differences can be reduced to the sumsets, in principle.
    To make the reverse implication we need in existence of an element $i\in A$ such that $i^2 = -1$.
    Sumsets and, more generally, sums of two different cosets of subgroups will be considered at  the next section.
\end{remark}

\bigskip

%Theorem \ref{t:A-A_in_G} has a simple consequence on the size of maximal arithmetic progression in multiplicative subgroup.
%Other results in the direction can be found in \cite{Sh_ineq}, e.g.

%\begin{corollary}
%    Let $\G \subseteq \f_p$ be a multiplicative subgroup, $|\G| <\sqrt{p}$.
%    Then for any arithmetic progression $P\subseteq \G$ one has $|P| \ll |\G|^{1/3+o(1)}$.
%\label{c:progr}
%\end{corollary}

To complete the proof of  Theorem \ref{t:A-A_total} from the introduction,
we need in a rather standard lemma, see e.g. the proof of Lemma 5 from \cite{Shparlinski_AD}.
%We
Let us
give a proof for the sake of completeness.

\begin{lemma}
    Let $\G \subseteq \f_p$ be a multiplicative subgroup, and $k$ be a positive integer.
    Then for any nonzero distinct elements $x_1,\dots,x_k$ one has
\begin{equation}\label{f:C_for_subgroups}
    |\G \bigcap (\G+x_1) \bigcap \dots \bigcap (\G+x_k)| = \frac{|\G|^{k+1}}{(p-1)^k} + \theta k 2^{k+3} \sqrt{p} \,,
\end{equation}
    where $|\theta| \le 1$.
\label{l:C_for_subgroups}
\end{lemma}
\begin{proof}
    Let $d=(p-1)/|\G|$.
    By $\chi_0$ denote the principal character.
    Then
$$
    \G(x) = \frac{1}{d} \sum_{\chi \in \mathcal{X}_d} \chi (x)
        = \frac{1}{d} \left( \chi_0 (x) + \sum_{\chi \in \mathcal{X}^*_d} \chi (x) \right) \,,
$$
where $\mathcal{X}_d = \{ \chi ~:~ \chi^d = \chi_0 \}$, $\mathcal{X}^*_d = \mathcal{X}_d \setminus \{ \chi_0 \}$.
Thus, putting $x_0 = 0$, we get
$$
    |\G \bigcap (\G+x_1) \bigcap \dots \bigcap (\G+x_k)|
        =
            \frac{1}{d^{k+1}} \sum_x \prod_{j=0}^k
                \left( \chi_0 (x+x_j) + \sum_{\chi \in \mathcal{X}^*_d} \chi (x+x_j) \right)
                =
$$
$$
    =
        \frac{|\G|^{k+1} ( p-(k+1))}{(p-1)^{k+1}}
            +
                 \frac{1}{d^{k+1}} \sum_{l=1}^{k+1}\,
                 %\binom{k+1}{l}
                 \sum_{\chi_{i_1}, \dots, \chi_{i_l} \neq \chi_0}\,
                    \sum_{x \neq 0} \chi_{i_1} (x+x_{i_1}) \dots \chi_{i_l} (x+x_{i_l})
                        =
                            \frac{|\G|^{k+1}}{(p-1)^k}
            +
                \sigma \,.
$$
By the well--known Weil's bound in Johnsen's form \cite{Johnsen}
$$
    \left| \sum_{x \neq 0} \chi_{i_1} (x+x_{i_1}) \dots \chi_{i_l} (x+x_{i_l}) \right| \le (l+1) \sqrt{p} + 1 \,,
$$
we get
$$
    |\sigma|
        \le
            \frac{k}{d^{k+1}} + \frac{\sqrt{p}}{d^{k+1}} \sum_{l=1}^{k+1}\, d^l \binom{k+1}{l}\, (l+2)
                \le
                    \sqrt{p} \sum_{l=0}^{k+1}\, \binom{k+1}{l}\, (l+2)
                        \le
                            \sqrt{p} k 2^{k+3}
                            %\,.
$$
as required.
$\hfill\Box$
\end{proof}

\bigskip

Lemma above immediately implies a corollary.

\begin{corollary}
    Let $A, \G \subset \mathbb{F}_p^{*}$ be two
    %sufficiently large
    multiplicative subgroups, $|\G| \ge p^{3/4}$.
%    $|A|<\sqrt{p}$ and $|\G| < p^{3/4}$.
    Suppose that $A-A \subseteq \xi \G \sqcup \{ 0 \}$  for some $\xi \in \mathbb{F}_p^{*}$.
    Then
    $$
        |A| \ll \min\left\{ \frac{|\G| \log^{1/2} p}{\sqrt{p}}, \sqrt{p} \right\} \,.
    $$
    Moreover, if a set $A'$ belongs to a shift of the subgroup $A$, $A'-A' \subseteq \xi \G \sqcup \{ 0 \}$
    and $|A| < \sqrt{p}$ then
    $
        |A'|^6 \ll |A|^4 \log |A| \cdot |\G|^2 p^{-1} \,.
    $
\label{cor:A-A_in_G}
\end{corollary}
\begin{proof}
Using previous lemma in the case $k=1$, combining with the arguments of the proof of Theorem \ref{t:A-A_in_G},
%and Corollary \ref{cor:strange},
and taking a half of a set by Lemma \ref{l:ab_p} if its needed,
we obtain
    $
        |A| \ll \frac{|\G| \log^{1/2} p}{\sqrt{p}} \,.
    $
It remains to recall a bound which gives us Lemma \ref{l:ab_p} for the size of an arbitrary $A$ with
$A-A \subseteq \xi \G \sqcup \{ 0 \}$ .
This completes the proof.
$\hfill\Box$
\end{proof}

\bigskip

Corollary above and Theorem \ref{t:A-A_in_G} give us Theorem \ref{t:A-A_total} from the introduction.

%\section{The proof}
\section{The proof of the main result}
\label{sec:proof}

The next lemma gives us an expression of $\T (A,B,C,D)$ via coefficients $c_\a$ of the sets $A,B,C,D$.

\begin{lemma}
    Let $\G \subseteq \f_q^*$ be a multiplicative subgroup.
    Let also $A,B,C,D \subseteq \f_p$ be sets, and $A-C, B-D \subseteq \G$.
    Then
\begin{equation}\label{f:E_c_k}
    \T (A,B,C,D) = |\G| \sum_{\a=1}^{|\G|} \sum_{c\in C} \sum_{d\in D} |c_\a (A-c)|^2 |c_\a (B-d)|^2 \,.
\end{equation}
\label{l:E_c_k}
\end{lemma}
\begin{proof}
A direct calculation (or see \cite{Sh_energy}, Lemma 8) shows that
$$
    \E^\times (A,B) = |\G| \sum_{\a=1}^{|\G|} |c_\a (A)|^2 |c_\a (B)|^2
$$
for any $A,B\subseteq \G$.
Using formula (\ref{f:T_energy}) and the fact that $A-c, B-d \subseteq \G$ for any $c\in C$, $d\in D$,
we get (\ref{f:E_c_k}).
This completes the proof.
$\hfill\Box$
\end{proof}

\bigskip

Using the eigenvalues technique, we can obtain a rather general result on differences inside multiplicative subgroups.

\begin{proposition}
    Let $A,\G\subseteq \f^*_p$ be multiplicative subgroups, a set $C$ belongs to a shift of $A$,
    $3 \le |C|$, $|A| <\sqrt{p}$.
    Let also $g : \f_q \to \R^{+}$ be an arbitrary even $\G$--invariant function.
    Suppose that $C-C \subseteq \xi \G \sqcup \{ 0 \}$ for some $\xi \in \mathbb{F}_p^{*}$.
    Then
\begin{equation}\label{f:weight_A}
    \left( \sum_x g(x) (C\c C) (x) \right)^2
        \ll
            \frac{|A|^4 \log |A|}{|C|^2 |\G|} \cdot \sum_x g^2 (x) (\G\c \G) (x) \,.
\end{equation}
%    Taking $g(x) = (\xi \G) (x)$ and assuming $|\G| < p^{3/4}$, we obtain
%\begin{equation}\label{f:weight_A_p}
%\end{equation}
\label{p:weight_A}
\end{proposition}
\begin{proof}
    Without loosing of generality we can assume that $\xi=1$.
%    Inequality (\ref{f:weight_A'}) follows from (\ref{f:weight_A})
%if one choose the weight $g(x)$ in an optimal way, namely, $g(x) = (A\c A) (x) / (\G \c \G) (x)$.
%    Thus, we need to check (\ref{f:weight_A}).
%
%
    Consider the operator $\oT^g_\G$ and denote by $\{ \mu_\a\}_{\a=1}^{|\G|}$ its eigenvalues.
    Note that for any $c\in C$ one has $C-c \subseteq \G \bigsqcup \{0\}$.
    Putting $C'_c = C\setminus \{ c\}$, we have $C'_c - c\subseteq \G$.
    Thus
$$
    \sum_x g(x) (C\circ C) (x) - 2 (g\circ C) (c)
        \le
            \langle \oT^g_\G (C'_c-c), C'_c-c \rangle = \sum_\a |c_\a (C'_c-c)|^2 \mu_\a \,.
$$
    Summing over $c\in C$ and using the condition $|C| \ge 3$, we obtain
\begin{equation}\label{tmp:02.05.2015_1}
    |C| \sum_x g(x) (C\circ C) (x)
        \ll
            \sum_\a \mu_\a \sum_{c\in C} |c_\a (C'_c-c)|^2 \,.
\end{equation}
    %Using
    Applying
    the Cauchy--Schwarz inequality, we get
$$
    |C|^2 \left( \sum_x g(x) (C\circ C) (x) \right)^2
        \ll
            \sum_\a |\mu_\a|^2 \cdot \sum_\a \sum_{c,\t{c} \in C} |c_\a (C'_c-c)|^2 |c_\a (C'_{\t{c}}-\t{c})|^2 \,.
$$
    %Applying
    Using
    formulas (\ref{f:trace2}), (\ref{f:T_invariance}), Lemma \ref{l:sum_of_squares}, as well as the arguments of the proof of Lemma \ref{l:E_c_k}, we get
$$
    |C|^2 \left( \sum_x g(x) (C\circ C) (x) \right)^2
        \ll
            |\G|^{-1} \sum_x g^{2} (x) (\G \c \G) (x)
                \cdot
                    \sum_{c,\t{c} \in C} \E^\times (C'_c-c, C'_{\t{c}}-\t{c})
                        \le
$$
$$
    \le
            |\G|^{-1} \sum_x g^{2} (x) (\G \c \G) (x)
                \cdot
                    \sum_{a,\t{a} \in A} \E^\times (A-a, A-\t{a})
                        =
                            |\G|^{-1} \sum_x g^{2} (x) (\G \c \G) (x) \cdot \T(A) \,.
$$
Finally, recalling Proposition \ref{p:sigma}, we obtain (\ref{f:weight_A}).
This completes the proof.
$\hfill\Box$
\end{proof}

\bigskip

Taking the weight $g(x)$ to be the characteristic function of the set $(-\xi \G) \bigcup \xi \G\bigsqcup \{ 0 \}$,
we get Theorem \ref{t:A-A_in_G} and Corollary \ref{cor:A-A_in_G}.

\bigskip

\begin{remark}
As we have seen the arguments of the proof of the proposition above allow to replace $A$ onto its a (large) subset $A'\subseteq A$ in spirit of Theorem \ref{t:A-A_in_G} and Corollary \ref{cor:A-A_in_G} of the previous section.
On the other hand there is
%The reason is an
an asymmetry
%of
between
$A$ and $\G$.
We use the group properties of $\G$ extensively but the only we need about $A$ is that $\T(A)$ is small.
For example, if $A$ is an arithmetic progression then our method works similarly.
\label{r:A_A'}
\end{remark}

\bigskip

Now we can consider the case of general sumsets in multiplicative subgroups.
We begin with a lemma which says that the average value of the action of an arbitrary operator $\oT^g_\G$
to multiplicative shifts of two functions can be
%expressed
calculated
%as an "action"\, to these functions
easily.
The crucial thing here that the weight $g$ is very general and does not require to be $\G$--invariant.

\begin{lemma}
    Let $\G\subseteq \f^*_q$ be a multiplicative subgroup.
    Let also $h_1,h_2$ be any functions with supports on $\G$ and $g : \f_q \to \C$ be an arbitrary function.
    Then
\begin{equation}\label{f:average_mult}
    \frac{1}{|\G|} \sum_{\gamma \in \G} \langle \oT_\G^g h^\gamma_1, h^\gamma_2 \rangle
        =
            \sum_{\a=1}^{|\G|} c_\a (h_1) \ov{c_\a (h_2)} \cdot \langle \oT_\G^g f_\a, f_\a \rangle \,.
\end{equation}
\label{l:average_mult}
\end{lemma}
\begin{proof}
We have $h_1(x) = \sum_\a c_\a (h_1) f_\a (x)$ and
$h_2(x) = \sum_\a c_\a (h_2) f_\a (x)$.
Thus by the orthogonality of the characters, we have
$$
    \sum_{\gamma \in \G} \langle \oT_\G^g h^\gamma_1, h^\gamma_2 \rangle
        =
            \sum_{\gamma \in \G} \sum_{x,y} \oT_\G^g (x,y) h_1 (\gamma x) \ov{h_2 (\gamma x)}
                =
$$
$$
                =
                \sum_{\a,\beta} c_\a (h_1) \ov{c_\beta (h_2)}
                \cdot \langle \oT_\G^g f_\a, f_\beta \rangle
                \cdot \left( \sum_{\gamma \in \G} \chi_\a (\gamma) \ov{\chi_\beta (\gamma)}  \right)
                    =
            |\G| \sum_{\a} c_\a (h_1) \ov{c_\a (h_2)} \cdot \langle \oT_\G^g f_\a, f_\a \rangle \,.
$$
as required.
%This completes the proof.
$\hfill\Box$
\end{proof}

\bigskip

Using lemma above we prove a general result on sumsets in multiplicative subgroups.

Having a set $Q\subseteq \f_p$ and a multiplicative subgroup $\G$ denote by $S_\G (Q)$ a minimal $\G$--invariant set containing $Q$.
Note that $S_\G (\xi Q) = S_\G (Q)$ for any nonzero $\xi$.
Clearly, $|S_\G (Q)| \le |\G Q| \le |\G| |Q|$.
Sometimes better estimates holds.
For example, if $Q = Q_1-Q_1$, where $Q_1 \subseteq \G$ then $|S_\G (Q)| \le |\G - \G|$.

\begin{proposition}
    Let $\G\subset \f^*_p$ be a multiplicative subgroup,
    % $|\G| \ge 3$,
    $A, B \subseteq \f_p$ be two sets.
    %, $3 \le |A| <\sqrt{p}$.
    Let also $g : S_\G (A-A) \to \R^{+}$ be an arbitrary even $\G$--invariant function.
    Suppose that $A-B \subseteq \G \sqcup \{ 0 \}$
    and
\begin{equation}\label{cond:sumsets_g}
    |B| \sum_x g(x) (A\c A) (x) \ge 3 \sum_x g(x) (A\c B) (x) \,.
\end{equation}
    Then
\begin{equation}\label{f:sumsets_g}
    \left( \sum_x g(x) (A\c A) (x) \right)^2
        \ll
            \frac{\T(A,B)}{|B|^2 |\G|} \cdot \sum_x g^2 (x) (\G\c \G) (x) \,.
\end{equation}
%    If $A-B \subseteq \G$ then inequality (\ref{f:sumsets_g}) holds without condition (\ref{cond:sumsets_g}).
\label{p:sumsets_g}
\end{proposition}
\begin{proof}
    For any $b\in B$ one has $A-b \subseteq \G \sqcup \{ 0 \}$.
    Putting $A'_b = A\setminus \{ b \}$, we have $A'_b - b \subseteq \G$.
    Applying Lemma \ref{l:average_mult} with $h_1=h_2 = A'_b - b$, we get for any $b\in B$
\begin{equation}\label{tmp:23.04.2015_2}
    \frac{1}{|\G|} \sum_{\gamma \in \G} \langle \oT_\G^g (A'_b - b)^\gamma, (A'_b - b)^\gamma \rangle
        =
            \sum_{\a=1}^{|\G|} |c_\a ((A'_b - b)^\gamma)|^2 \cdot \langle \oT_\G^g f_\a, f_\a \rangle \,.
\end{equation}
Summing over $b\in B$
%, using condition (\ref{cond:sumsets_g})
and
using $\G$--invariance of $g$
as well as the arguments of the proof of Proposition \ref{p:weight_A},
we see that the left--hand side of (\ref{tmp:23.04.2015_2}) is
$$
    \frac{1}{|\G|} \sum_{\gamma \in \G} \sum_{b\in B}
        \sum_{x,y} g(x-y) (A'_b - b)^\gamma (x) (A'_b - b)^\gamma (y)
    =
$$
$$
    =
    \frac{1}{|\G|} \sum_{\gamma \in \G} \sum_{b\in B} \sum_{x,y} g(x-y) A'_b (x\gamma + b\gamma) A'_b (y\gamma + b\gamma) \G(x) \G(y)
    \ge
$$
$$
    \ge
            \frac{1}{|\G|} \sum_{\gamma \in \G}
    \left( |B| \sum_x g(x) (A^\gamma \c A^\gamma) (x) - 2 \sum_{b\in B} \sum_x A(x\gamma + b \gamma)
        g(x - b \gamma^{-1} + b) \right)
        \ge
$$
$$
    \ge
%    \frac{1}{|\G|} \sum_{\gamma \in \G} \left(
    |B| \sum_x g(x) (A \c A) (x) -
    %2 \sum_{x} g(x) (A\c B) (x)
    \frac{2}{|\G|} \sum_{\gamma \in \G} \sum_{b\in B} \sum_{x} g(x) A(x\gamma + b)
%    \right)
%        \ge
%$$
%$$
    \ge
$$
\begin{equation}\label{tmp:23.04.2015_3}
    \ge
        |B| \sum_x g(x) (A \c A) (x) - 2 \sum_{x} g(x) (A\c B) (x)
%            \ge
    \ge
    3^{-1} |B| \sum_x g(x) (A \c A) (x) \,.
\end{equation}
Here we have used
%the
condition (\ref{cond:sumsets_g}).
%%$|\G| \ge 3$.
%The last estimate follows from the following bound
%\begin{equation}\label{cond:sumsets_g}
%    |B||\G| \sum_x g(x) (A\c A) (x) \ge 4 \sum_{a\in A} (g \c A) (a) \sum_{\gamma\in \G} B^\gamma (a) \,,
%\end{equation}
%which is an easy consequence of the condition $|\G| \ge 4$.
Thus
$$
        |B| \sum_x g(x) (A \c A) (x)
    \ll
        \sum_\a \langle \oT_\G^g f_\a, f_\a \rangle \sum_{b\in B} |c_\a ((A'_b - b)^\gamma)|^2 \,.
$$
Applying the Cauchy--Schwarz inequality and Lemmas \ref{l:convex_eigenvalues}, \ref{l:sum_of_squares}, we obtain
$$
    |B|^2  \left( \sum_x g(x) (A\c A) (x) \right)^2
        \ll
            \sum_\a \langle \oT_\G^g f_\a, f_\a \rangle^2 \cdot
                \sum_\a \sum_{b,\t{b} \in B} |c_\a ((A'_b - b)^\gamma)|^2 |c_\a ((A'_{\t{b}} - \t{b})^\gamma)|^2
                    \le
$$
$$
                    \le
                        |\G|^{-1} \sum_\a |\mu_\a (\oT_\G^g)|^2 \cdot
                            \sum_{b,\t{b} \in B} \E^{\times} (A-b,A-\t{b})
                                \le
                                    |\G|^{-1} \sum_x g^2 (x) (\G\c \G) (x) \cdot \T(A,B)
                \,.
$$
This concludes the proof.
$\hfill\Box$
\end{proof}

\bigskip

Note that the condition on $\G$--invariance of $g$ in Proposition \ref{p:sumsets_g}
is not very important, it needs just for
easiest way to obtain estimate
%calculation of the sum
(\ref{tmp:23.04.2015_3}).

\bigskip

Now we are able to
%obtain
prove
the main result of the section.

%\bigskip

\begin{theorem}
    Let $A,\G\subset \f^*_p$ be multiplicative subgroups,
    %$|\G| \ge 3$,
    $|A| <\sqrt{p}$ and $C,D \subseteq \f^*_p$ be arbitrary sets.
    Suppose that for some $\xi, \eta \in \f_p^*$ the following holds $C\subseteq \xi A+s$,
    $D \subseteq \eta A+s$, $s\in \f_p$, $|C| \ge 3$ and
    $$
        C - D \subseteq \G \bigsqcup \{ 0 \} \,.
    $$
%    Put $S=\xi \G \cdot (C-C)$.
    Then
    \begin{equation}\label{f:sumsets_xi_eta}
        |C|^8 |D|^4 |\G|^2 \ll |A|^8 |S_\G (C-C)| \E^{+} (\G) \log^2 |A| \,.
    \end{equation}
    If $|\G| \ll p^{3/5 - o(1)}$ then $|C|^4 |D|^2 \ll |A|^4 |S_\G (C-C)|^{2/3} \log |A|$.
\label{t:sumsets_xi_eta}
\end{theorem}
\begin{proof}
    Put
    %$S=S_\G (\xi (C-C))$.
    $S=S_\G (C-C)$.
    Take $A=C$, $B=D$, $g (x) = S(x)$ and apply Proposition \ref{p:sumsets_g}.
    Since $C-C \subseteq S$ and $|C| \ge 3$ it follows that
$$
    |C|^2 |D| = \sum_{x} g(x) (C\c C) (x) \ge 3 |C| |D| \ge  \sum_{x} g(x) (C\c D) (x) \,.
$$
    Thus condition (\ref{cond:sumsets_g}) of Proposition \ref{p:sumsets_g} holds and whence
$$
    |C|^4 \ll \frac{\T(\xi A + s, \eta A + s)}{|D|^2 |\G|} \sum_{x} S(x) (\G \c \G) (x) \,.
$$
    By the invariance (\ref{f:T_invariance}) and Proposition \ref{p:sigma},
%    considering large subsets of $A$ (and hence large subsets of $C$ and $D$) if it is need (see Remark \ref{r:A_A'}),
    we have
    %with some abuse of the notation that
\begin{equation}\label{tmp:24.04.2015_1}
    \T(\xi A +s, \eta A + s) = \T(\xi A, \eta A) \ll |A|^4 \log |A| \,.
\end{equation}
    Further, applying the Cauchy--Schwarz inequality, we get
\begin{equation}\label{tmp:24.04.2015_2}
    \sigma^2 := \left( \sum_{x} S(x) (\G \c \G) (x) \right)^2
        \le
            |S| \sum_{x} S(x) (\G \c \G)^2 (x)
                \le
                    |S| \E^{+} (\G) \,.
\end{equation}
Another bound for $\sigma$ follows from Lemma \ref{l:mitkin} or see Corollary 6 from \cite{Sh_average}
\begin{equation}\label{tmp:24.04.2015_2'}
    \sigma \ll |S|^{2/3} |\G| \,,
\end{equation}
provided $|\G|^3 |S| \ll p^3$.
But
$$
    |S| \le |C-C| |\G| \le |C|^2 |\G| \ll |\G|^{2+o(1)}
$$
by Theorem \ref{t:Shparlinskii}.
Hence, the assumption $|\G| \ll p^{3/5 - o(1)}$ implies $|\G|^3 |S| \ll p^3$.
Combining (\ref{tmp:24.04.2015_1}) and (\ref{tmp:24.04.2015_2}), we get
$$
    |C|^8 |D|^4 |\G|^2 \ll |A|^8 |S| \E^{+} (\G) \log^{2} |A| \,.
$$
Using (\ref{tmp:24.04.2015_1}) and (\ref{tmp:24.04.2015_2'}), we obtain
$|C|^4 |D|^2 \ll |A|^4 |S|^{2/3} \log |A|$.
This completes the proof.
$\hfill\Box$
\end{proof}

\bigskip

Theorem \ref{t:sumsets_xi_eta} implies an important corollary which beats the exponent $\frac{1}{2}$ of Theorem \ref{t:Shparlinskii} in a particular case when $A = \xi A$, $B=\eta A$ and $A$ is a subgroup belonging to $\G$.
%
%
%The last result was proved in \cite{Sh_ineq}.
We need in a result from \cite{Sh_ineq}.

\begin{theorem}
    Let $p$ be a prime number and $\G \subset \F_p^{*}$ be a multiplicative subgroup, $|\G| < p^{\frac{1}{2}} \log^{-\frac{1}{5}} p$.
    Then
\begin{equation}\label{f:32/13}
    \E^{+} (\G) \ll |\G|^{\frac{32}{13}} \log^{\frac{41}{65}} |\G| \,.
\end{equation}
\label{t:32/13}
\end{theorem}

\begin{corollary}
    Let $A,\G\subset \f^*_p$ be multiplicative subgroups, $A\subseteq \G$ and
    %$3 \le
    $|\G| < p^{\frac{1}{2}} \log^{-\frac{1}{5}} p$.
    Suppose that for some $\xi, \eta \in \f_p^*$ the following holds
    $$
        \xi A+ \eta A \subseteq \G \bigsqcup \{ 0 \} \,.
    $$
    Then
\begin{equation}\label{f:A+A}
    |A| \ll |\G|^{\frac{19}{39}} \log^{\frac{57}{65}} |\G| \,.
\end{equation}
\label{cor:A+A}
\end{corollary}
\begin{proof}
    One can assume that $|A| \ge 3$ because otherwise the result is trivial.
    Applying formula (\ref{f:sumsets_xi_eta}) of Theorem \ref{t:sumsets_xi_eta} for $C=\xi A$, $D=-\eta A$,
    combining  it with Theorem \ref{t:32/13},  we get
\begin{equation}\label{tmp:24.04.2015_3}
    |A|^4 \ll |S| |\G|^{\frac{6}{13}} \log^{\frac{171}{65}} |\G| \,,
\end{equation}
    where $S=S_\G (\xi (A-A))$.
    It remains to estimate the size of $S$.
    We have $A-A = \bigsqcup_{j=1}^t \xi_j A$, $t= |A-A|/|A| \le |A|$.
    By assumption $A\subseteq \G$. It follows that $\G = \bigsqcup_{i=1}^s \eta_i A$, $s=|\G|/|A|$.
    Whence
\begin{equation}\label{f:S_1.5}
    |S| \le |\bigcup_{i=1}^s \bigcup_{j=1}^t \eta_i \xi_j A| \le st |A| \le |A| |\G| \,.
\end{equation}
    Substituting the last bound into (\ref{tmp:24.04.2015_3}), we obtain the required bound for the size of $A$.
    This concludes the proof.
$\hfill\Box$
\end{proof}

\bigskip

Now we can
%improve
%%Finally, we
refine
%Corollary
Proposition
\ref{p:subgroup_decomp_impossible} from the introduction.

\begin{corollary}\label{cor:subgroup_decomp_impossible_new}
Let $\eps \in (0,1]$ be a real number.
Let also $A, \G \subset \mathbb{F}_p^{*}$ be two sufficiently large multiplicative subgroups,
%$|\G| \le p^{1-\eps}$,
$|\G| \le p^{\frac{1}{2}} \log^{-\frac{1}{5}} p$.
and $B \subseteq \mathbb{F}_p$ be an arbitrary nonempty set.
Then the sets $\G$, $\G \bigsqcup \{ 0 \}$
has no nontrivial representations as $\xi A+B$ for any $\xi \in \f_p$.
%are primitive.
\end{corollary}
\begin{proof}
%Suppose that $\G=A+B$ or $\G \bigsqcup \{ 0 \} = A+B$, $|A|,|B| > 1$.
%In the first case $|\G| \ge 3$ by the Cauchy--Davenport inequality.
%In the second case if $|\G| =2$ then $\G = \pm 1$, $A=\pm 1$ and we get a contradiction.
%Thus, without loosing of generality one can assume that $|\G| \ge 3$.
%
%
For simplicity assume that $\xi=1$, otherwise the
%arguments are similar.
the proof is similar.
We repeat the arguments from \cite{SSV}, so let us miss some details.
Consider a multiplicative subgroup $H=\G \cap A$, put $B_\xi = B\cap \xi H$
and denote by $k$ the number of nonempty sets $B_\xi$.
Take $b_j \in B_{\xi_j}$, $j\in [k]$.
Thus $b_j$ belong to different cosets relatively to $H$.
Using Lemma \ref{l:mitkin}, the assumption
$|\G| < p^{\frac{1}{2}} \log^{-\frac{1}{5}} p$
%$|\G| \le p^{\frac{1}{2}-\eps}$
and the fact that $\G, \G \bigsqcup \{ 0 \} = A+B$ and hence $|A| \ll |\G|^{1/2+o(1)}$, we get
\begin{equation}\label{tmp:17.05.2015_1-}
    k |A| = \sum_{j=1}^k (A \circ \G) (b_j) \ll (|\G| |A| k^2)^{1/3}
    %\,.
\end{equation}
without any further restrictions on $A$ and $\G$.
%Whence
%one can show that
Inequality (\ref{tmp:17.05.2015_1-}) gives us
$k \ll |\G| / |A|^2$.
%$k\ll |\G|^\eps$.
By the pigeonhole principle there is $\xi$ such that $|B_\xi| \ge |B|/k$.
We have $H\subseteq A$, and $B_\xi \subseteq B,H$. Hence $H+B_\xi \subseteq  \G, \G \bigsqcup \{ 0 \}$.
In view of Proposition \ref{p:subgroup_decomp_impossible} and our assumption that
$A, \G \subset \mathbb{F}_p^{*}$ are two sufficiently large multiplicative subgroups,
 we obtain that $H$ is also sufficiently large and $|H|\ge 3$, in particular.
Applying formula (\ref{f:sumsets_xi_eta}) of Theorem \ref{t:sumsets_xi_eta}
with $A=H$, $C=H$, $D=B_\xi$, $\xi = \eta =1$, $s=0$
and
using
the upper bound for $k$, we obtain
\begin{equation}\label{tmp:17.05.2015_1}
    |H|^8 |B|^4 |\G|^2 \ll k^4 |A|^8 |S_\G (H-H)| \E^{+} (\G) \log^2 |\G|
        \ll
            |\G|^4
            %|\G|^{4\eps} |A|^8
            |S_\G (H-H)| \E^{+} (\G) \log^2 |\G| \,.
\end{equation}
By the calculations of  Corollary \ref{cor:A+A}, see formula (\ref{f:S_1.5}),
%(or just use the fact that $H\subseteq \G$),
we know that
$|S_\G (H-H)| \le |H| |\G|$.
Because of
$k \ll |\G| / |A|^2$
%$k \ll |\G|^{\eps}$
and, trivially, $k \ge |B|/|H|$, we get
%$|H| \gg |B| |\G|^{-\eps}$.
$|H| \gg |B| |A|^2 |\G|^{-1}$.
Substitution the last estimates into (\ref{tmp:17.05.2015_1}) gives us
$$
    |B|^{11} %|\G|^{}
    |A|^{14}
    \ll
    %|\G|^{11\eps} |A|^8
    |\G|^{10}
    \E^{+} (\G) \log^2 |\G| \,.
$$
Finally, by Theorem \ref{t:Shparlinskii}, we know that $|A|,|B| \gg |\G|^{1/2 - o(1)}$.
Combining the last bound with Theorem \ref{t:32/13}, we arrive to a contradiction.
This completes the proof.
$\hfill\Box$
\end{proof}

\bigskip

Using the full power of
%new
the
upper bound for the additive energy of a multiplicative subgroup from \cite{Sh_energy}
instead of (\ref{f:32/13}) one can  refine the restriction
%$|\G| \le p^{\frac{1}{2}-\eps}$.
$|\G| \le p^{\frac{1}{2}} \log^{-\frac{1}{5}} p$.

%\bigskip

%First of all  let us make a general remark on additive decompositions.
%If $A+B \subseteq S$ then shifting $B$ we can suppose that $0\in B$ and hence $A\subseteq S$.
%So, we will suppose, usually, that if a set $S$ is additively decomposable,
%$S=A+B$, $|A|, |B| \ge 2$ then $A\subseteq S$.

%\bigskip

\section{On Lev--Sonn's problem}
\label{sec:CR}

%\bigskip

Lev--Sonn's problem on representation of the set of quadratic residues \cite{LS} (see also \cite{BMR}) was discussed in the introduction.
In the section we consider a general case of an arbitrary subgroups.
%Let us note
First of all let us
%make
derive
one more consequence of Proposition \ref{p:weight_A}.
%As we discussed in the introduction, the case $\G$ equals the quadratic residues of corollary below was studied in %\cite{LS} (see also \cite{BMR}).

\begin{corollary}
    Let $A, \G \subset \mathbb{F}_p^{*}$ be two
    multiplicative subgroups.
    %, $|A| >1$.
    Suppose that $A-A = \xi \G \sqcup \{ 0 \}$ for some $\xi \in \mathbb{F}_p^{*}$
    and $(A\c A) (x) = c$ is a constant onto $\xi \G$.
    Then $|A|^2-|A| = c|\G|$ and either $|\G| = O(1)$ or $|\G| \gg \frac{p}{\log p}$.
    If $|\G| \gg \frac{p}{\log p}$ then
\begin{equation}\label{f:pds_1}
    \E^{+} (A) \ll \frac{|\G|}{p} \cdot |A|^2 \log |A| \,,
\end{equation}
    and
\begin{equation}\label{f:pds_2}
    c^2 \ll \frac{|A|^2 \log |A|}{p} \,.
\end{equation}
\label{c:CR}
\end{corollary}
\begin{proof}
Because of $\G$ is nonempty we can assume that $|A|>1$.
The fact that $|A|^2-|A| = c|\G|$ follows from trivial calculations.
By  our assumption $(A\c A) (x)$ is constant onto $\G$.
In the situation one can choose the weight $g(x)$ in an optimal way, namely, $g(x) = (A\c A) (x) / (\G \c \G) (x)$.
Clearly, $g(x)$ is an even $\G$--invariant nonnegative function.
Thus, applying Proposition \ref{p:weight_A} with $C=A$, we obtain
\begin{equation}\label{f:weight_A'}
    \sum_x \frac{(A\c A)^2 (x)}{(\G \c \G) (x)}
        \ll
            \frac{|A|^2 \log |A|}{|\G|} \,,
\end{equation}
provided by $|A| <\sqrt{p}$ and $|A| \ge 3$.
The last inequality trivially takes place because otherwise $|\G| < |A|^2 \ll 1$.
If $|\G| < p^{3/4}$ then by Theorem \ref{t:Shparlinskii}
one can assume that $|A| < \sqrt{p}$ for sufficiently large $\G$.
Using (\ref{f:weight_A'}), Lemma \ref{l:mitkin} and the inequality  $|A|>1$, we get
$$
    2^{-1} |A|^2 \le |A|^2 - |A| \le \sum_{x\neq 0} (A\c A)^2 (x) \ll |A|^2 \log |A| \cdot |\G|^{-1/3} \,.
$$
Because of by Theorem \ref{t:Shparlinskii}, we have $|A| \ll |\G|^{1/2+o(1)}$
(or just use a simple bound $|A| \le |\G|$)
%we arrives at a contradiction.
it gives us a contradiction for sufficiently large subgroup $\G$.

If $|\G| \ge p^{3/4}$ then using Lemma \ref{l:C_for_subgroups}
and taking a half of the  set $A$ in view of Lemma \ref{l:ab_p} if its needed,
we get by (\ref{f:weight_A'})  and the previous calculations that
\begin{equation}\label{tmp:23.04.2015_1}
    3^{-1} \E^{+} (A) \le \sum_{x\neq 0} (A\c A)^2 (x) \ll \frac{|\G|}{p} \cdot |A|^2 \log |A| \,.
\end{equation}
Because of $\E^{+} (A) \gg |A|^2$ we see from the previous estimate that $|\G| \gg p/\log p$.
We have obtained (\ref{f:pds_1}) already and to  get (\ref{f:pds_1}),
one can insert the condition $(A\c A) (x)=c$, $x\in \xi \G$ into (\ref{tmp:23.04.2015_1}).
This completes the proof.
$\hfill\Box$
\end{proof}

\bigskip

Note that for small
%$A$
$\G$
it can be
%even
$\xi \G \bigsqcup \{ 0 \} = A-A$.
For example (see \cite{LS}),
%$p=3$, $\G = \{ 1 \}$, $A=\{ 0, 1\}$
$p=5$, $A = \{ -1, 1\}$, $\xi \G=2 \cdot \{ -1 , 1\}$,
and $p=13$,
$\G = \{ 1, 3, 4, 9, 10, 12 \}$,
%$A=\{ 0, 1, 4 \}$.
$A=\{ 2,5,6\} = 2\cdot \{1,3,9\}$.
%
%
%
%We will give a full list of subgroups $A, \G$ of such a form
%$$
%    {\bf ????}
%$$
%Moreover,
Actually,
it is easy to see that for any subgroup $A$ of order $2,3$ one has
$\xi \G \bigsqcup \{ 0 \} = A-A$  for some $\xi$, and $|\G| = |A|^2 - |A|$.
Thus, the case $|\G| = O(1)$ is possible in the corollary above.

%\begin{proposition}
%    Let $A, \G \subset \mathbb{F}_p^{*}$ be two multiplicative subgroups.
%    Suppose that $A-A = \xi \G \sqcup \{ 0 \}$ for some $\xi \in \mathbb{F}_p^{*}$
%    and $(A\c A) (x) = c$ is a constant onto $\G$.
%    Then either $|\G| \ge \frac{...p}{\log p}$
%    or $\G$, $A$ belong to the list ......
%\label{p:CR}
%\end{proposition}
%\begin{proof}
%$\hfill\Box$
%\end{proof}

\section{The quantity $\T(A)$ and concluding remarks}
\label{sec:T_A_f_p}

In the section we discuss further properties of the quantity $\T(A)$.

%\bigskip

First of all, let us prove a simple general upper bound for $\T(A,B,C,D)$, where $A,B,C,D$ are subsets of an arbitrary finite field.
Bound (\ref{f:sigma_l}) below is tight, as the example $q=p^2$ and the case $A=B=C=D$ is a subfield
isomorphic to $\f_p$ shows.

\begin{proposition}
Let $q$ be a prime power, $A$ be a subgroup of $\mathbb{F}_q^*$.
Then for any sets $B,C,D\subseteq \f_q$ one has
\begin{equation}\label{f:sigma_l}
    \T (A, B,C,D)
        \le
            \frac{|A|^2 |B|^2 |C| |D|}{q-1} +  |B| |C| |D| q + \T(A,B,\{0\},D) \,.
\end{equation}
\label{p:sigma_l}
\end{proposition}
\begin{proof}
Using formula (\ref{f:T_energy}), we get
$$
    \T (A, B,C,D) = \frac{1}{q-1} \sum_{c\in C} \sum_{d\in D} \sum_\chi |(A-c)\,\t{}\, (\chi)|^2 |(B-d)\,\t{}\, (\chi)|^2 \,,
$$
where the summation is taken over all Dirichlet characters $\chi$.
The principal character gives us the term $\frac{|A|^2 |B|^2 |C| |D|}{q-1}$.
It is well--known, see e.g. \cite{ES} that for all non--principal $\chi$, we get
$$
    |(A-c)\,\t{}\, (\chi)|^2 := |\sum_{x} A(x+c) \ov{\chi(x)}|^2 =
        |A|^{-2} |\sum_{x} \sum_{y\in A} A(xy+c) \ov{\chi(x)}|^2
            =
$$
$$
    =
        |A|^{-2} |\sum_{a\in A} \sum_{y\in A} \chi(a-cy)|^2
            \le
                q \,,
$$
provided by $c\neq 0$.
Using the Parseval's identity, and combining all bounds, we obtain (\ref{f:sigma_l}).
This completes the proof.
$\hfill\Box$
\end{proof}

\bigskip

In the case of the prime field one can obtain a simple nontrivial upper bound for the quantity $\T(A)$
which is better than (\ref{f:sigma_l}) for small sets.
The arguments follows \cite{Green_sum-prod}.

Let us start with an easy combinatorial lemma.

\begin{lemma}
Let $X$ be a finite set, $|X|= n$, $A_j \subseteq X$, $j\in [m]$ be a collection of subsets of $X$, $|A_j| \ge \d n$,
$\d \in (0,1]$.
Suppose that $m\ge 2/\d$.
Then there is a pair $(i,j)$, $i\neq j$ such that $|A_i \cap A_j| \ge 2^{-1} \d^2 n$.
\label{l:combinatorial}
\end{lemma}
\begin{proof}
We have
$$
    \d mn \le \sigma:= \sum_{j=1}^m |A_j| = \sum_{x\in X} \sum_{j=1}^m A_j (x) \,.
$$
Using the Cauchy--Schwarz inequality, we get
$$
    (\d mn)^2
        \le
            \sigma^2
        \le
            n \sum_{x\in X} \sum_{i,j=1}^m A_i (x) A_j (x)
                =
    n \sum_{i,j=1}^m |A_i \cap A_j|
        =
            n \left( \sigma + \sum_{i \neq j} |A_i \cap A_j| \right) \,.
$$
Applying the assumption $m\ge 2/\d$, we obtain
$$
    2^{-1} \d^2 n^2 m^2 \le 2^{-1} \sigma^2 \le n \sum_{i \neq j} |A_i \cap A_j|
$$
as required.
%This completes the proof.
$\hfill\Box$
\end{proof}

\bigskip

The next proposition is a "dual"\, version of the sum--products estimate, where, traditionally,
the quantity $\sum_{c\in C} \E^{+} (cA,A)$ is considered, see e.g. \cite{Green_sum-prod}.

\begin{proposition}
    Let $p$ be a prime number and $A,B,C \subseteq \F_p$ be three sets,
    $|A| \le \sqrt{p}$,
    %or $|B|\le \sqrt{p}$,
    and $|C| \gg p^\d$, where $\d >0$ be a fixed
    %constant.
    number.
    Then there is an absolute constant $\eps = \eps(\d)>0$ such that
\begin{equation}\label{f:semi_T}
    \sum_{c\in C} \E^\times (A-c,B) \ll |A|^{3/2} |B|^{3/2} |C| p^{-\eps} \,.
\end{equation}
\label{p:semi_T}
\end{proposition}
\begin{proof}
Put
\begin{equation}\label{tmp:10.07.2015_1}
    \sum_{c\in C} \E^\times (A-c,B) = \frac{|A|^{3/2} |B|^{3/2} |C|}{K} \,,
\end{equation}
where $K\ge 1$ is some parameter.
We need to
%bound the number $K\ge 1$ from below.
obtain a lower bound for the number $K$ of the form $K\gg p^{\eps}$, where $\eps = \eps (\d)>0$
is some absolute constant.
From formulas (\ref{f:E_CS}),
(\ref{tmp:10.07.2015_1}) it follows that there is a set $C'\subseteq C$, $|C'| \ge |C|/(2K)$
such that for all $c\in C'$ one has $\E^\times (A-c,B) \ge |A|^{3/2} |B|^{3/2} / (2K)$.
Applying the Balog--Szemer\'{e}di--Gowers Theorem \ref{t:BSzG_A-B}, we find for any $c\in C'$ a set $H_c$ with
$|H_c H_c| \ll K^{M} |H_c|$, further, $|A|, |B| \ll K^{M} |H_c|$ and
$|(A-c) \cap x_c H_c|, |B \cap y_c H_c| \gg K^{-M} |H_c|$.
Here $M>0$ is an absolute constant.
% and $x=x(c)$, $y=y(c)$.
Put $H'_c = A\cap (x_c H_c+c)$, $H''_c = B \cap y_c H_c$ and  apply Lemma \ref{l:combinatorial} with
$X= A\times B$ and the family of sets $\{ H'_c \times H''_c\}_{c\in C'}$.
By the lemma and the assumption $|C| \gg p^{\d}$, we find $c_1,c_2 \in C'$, $c_1 \neq c_2$ such that the sets
$H' = H'_{c_1} \cap H'_{c_2}$, $H'' = H''_{c_1} \cap H''_{c_2}$ have sizes at least $K^{-M_1} |A|$, $K^{-M_1} |B|$,
respectively.
Here $M_1 >0$ is another absolute constant.
We have
\begin{equation}\label{tmp:11.07.2015_1}
    (H'-c_1) (H'-c_2) \subseteq x_{c_1} x_{c_2} H_{c_1} H_{c_2} \,.
\end{equation}
Applying the Pl\"{u}nnecke inequality, see e.g. \cite{TV}, we obtain
$$
    K^{-M_1} |B| |H_{c_1} H_{c_2}|
        \ll
    |H''| |H_{c_1} H_{c_2}| \le |H_{c_1} H''| |H_{c_2} H''| \le |H_{c_1} H''_{c_1}| |H_{c_2} H''_{c_2}|
        \le
$$
$$
        \le
             |H_{c_1} H_{c_1}| |H_{c_2} H_{c_2}|
                \ll
                    K^{2M} |H_{c_1}| |H_{c_2}|
                        \ll
                            K^{4M} |A| |B| \,.
$$
Hence $|H_{c_1} H_{c_2}| \ll K^{4M+M_1} |A|$ and thus by inclusion (\ref{tmp:11.07.2015_1}), we get
\begin{equation}\label{tmp:11.07.2015_2}
    |(H'-c_1) (H'-c_2)| \ll K^{4M+M_1} |A| \ll K^{4M+2M_1} |H'| \,.
\end{equation}
%Suppose that $c_1 \neq 0$.
Put $A_*=H'-c_1$, $B_*=H'-c_2$, $C_*=H'-c_1$, $d=c_1-c_2$.
%If $c_1=0$ then put $d=c_2 \neq 0$ and $A=H'-c_2$, $B=H'-c_1=H'$, $C=H'$.
Because of $c_1 \neq c_2$, we
%get
see that
$d\neq 0$, further, the sets $A_*,B_*,C_*$ have the same size $|H'|$,
and by (\ref{tmp:11.07.2015_2}) one has $|A_* B_*| \ll K^{4M+2M_1} |H'|$.
Further, $|(A_*+d)C_*| = |(H'-c_2)(H'-c_1)|$ and again by (\ref{tmp:11.07.2015_2})
the last quantity is bounded as $O(K^{4M+2M_1} |H'|)$.
Since $|A| \le \sqrt{p}$,
% or $|B| \le \sqrt{p}$,
we obtain $|H'| \le \sqrt{p}$.
Applying Theorem \ref{t:Z} with $A=A_*$, $B=B_*$, $C=C_*$, we arrive to a contradiction for sufficiently small $K$.
This completes the proof.
$\hfill\Box$
\end{proof}

\bigskip

Proposition above has an immediate consequence.

\begin{corollary}
    Let $p$ be a prime number and $A,B,C,D \subseteq \F_p$ be three sets,
    $|A| \le \sqrt{p}$,
    or $|B| \le \sqrt{p}$,
    and $|C| \gg p^\d$ or $|D| \gg p^\d$,  where $\d >0$ be a fixed
    %constant.
    number.
    Then there is an absolute constant $\eps = \eps(\d)>0$ such that
\begin{equation}\label{f:semi_T_c}
    \T (A,B,C,D) \ll |A|^{3/2} |B|^{3/2} |C| |D| p^{-\eps} \,.
\end{equation}
\label{c:semi_T+}
\end{corollary}

\bigskip

Now
%Finally
we obtain
%a consequence
an analog
of Propositions \ref{p:weight_A}, \ref{p:sumsets_g} about the intersections of sumsets
and multiplicative subgroups.
We thanks to Dmitry Zhelezov who asked us about possible generalizations of our results in this direction.

\begin{proposition}
    Let $A,\G\subset \f^*_p$ be multiplicative subgroups,
    %$|\G| \ge 4$,
    $|A| <\sqrt{p}$ and $C,D \subseteq \f^*_p$ be an arbitrary sets.
    Suppose that for some $\xi, \eta \in \f_p^*$, $s\in \f_p$ the following holds $C\subseteq \xi A + s$,
    $D \subseteq \eta A+s$ and
    put $S=S_\G (C-C)$.
    Then
    \begin{equation}\label{f:sumsets_intersections01}
        \left( \sum_{x\in \G} (D \c C) (x) \right)^8
            \ll
                |D|^4 |\G|^{-2} |A|^8 |S| \E^{+} (\G) \log^2 |A| \,.
    \end{equation}
    If $|\G| \le p^{3/5 - o(1)}$ then
    \begin{equation}\label{f:sumsets_intersections02}
        \left( \sum_{x\in \G} (D \c C) (x) \right)^4
            \ll
                |D|^2 |A|^4 |S|^{2/3} \log |A| \,.
    \end{equation}
\label{p:sumsets_intersections}
\end{proposition}
\begin{proof}
%    Let $l=\log |A|$.
    Put
$$
    \sigma := \sum_{x\in \G} (D \c C) (x) = \sum_{x\in D} |\G \cap (C-x)| \,.
$$
%    By the pigeonhole principle the is a real number $\D$, $\D \ge 2^{-1} \sigma |D|^{-1}$
%    and a set $D'\subseteq D$ such that for any $x\in D'$ one has
%    $\D < |\G \cap C-x| \le 2\D$ and
    %Put
    Take
    $D' := \{ x\in D ~:~ |\G \cap (C-x)| \ge 2^{-1} \sigma |D|^{-1} \}$.
    Then
\begin{equation}\label{tmp:25.04.2015_-1}
    \sum_{x\in D'} |\G \cap (C-x)| \ge 2^{-1} \sigma \,.
\end{equation}
    Applying the Cauchy--Schwarz inequality, we get
\begin{equation}\label{tmp:25.04.2015_1}
    \sum_{x\in D'} |\G \cap (C-x)|^2 \gg \sigma^2 |D|^{-1} \,.
\end{equation}
    Returning to (\ref{tmp:25.04.2015_-1}), we see that for any $x \in D'$ there is a set $C_x \subseteq C$ such that $C_x-x\subseteq \G$.
    In view of (\ref{tmp:25.04.2015_1}) it follows that
$$
    \sum_{x\in D'} \langle \oT^{S}_\G (C_x-x), C_x-x \rangle = \sum_{x\in D'} |C_x|^2 \gg \sigma^2 |D|^{-1} \,.
$$
    Using the arguments as in the proofs of Propositions  \ref{p:weight_A}, \ref{p:sumsets_g}, combining with the Cauchy--Schwarz inequality and Lemma \ref{l:sum_of_squares}, we obtain
$$
    \sigma^4 |D|^{-2} \ll \left( \sum_{\a} \mu_\a (\oT^{S}_\G) \sum_{x\in D'} \langle C_x-x, f_\a \rangle^2 \right)^2
        \ll
$$
$$
        \ll
            \sum_{x\in S} (\G \c \G) (x) \cdot \sum_\a \sum_{x,x'\in D'} \langle C-x, f_\a \rangle^2
            \langle C-x', f_\a \rangle^2
            = |\G|^{-1} \sum_{x\in S} (\G \c \G) (x) \cdot \T(D,C) \,.
$$
    By the assumption $C\subseteq \xi A+s$ and $D \subseteq \eta A+s$.
    Applying Proposition \ref{p:sigma}, we obtain
$$
    \sigma^4 \ll |\G|^{-1} |D|^2 |A|^4 \log |A| \cdot \sum_{x\in S} (\G \c \G) (x) \,.
$$
    As in the proof of Theorem \ref{t:sumsets_xi_eta} one can estimate the sum $\sum_{x\in S} (\G \c \G) (x)$
    in two different ways as $(|S| \E^{+} (\G))^{1/2}$ and $|S|^{2/3} |\G|$, provided $|\G| \le p^{3/5-o(1)}$.
    This completes the proof.
$\hfill\Box$
\end{proof}

\begin{example}
    Let $\xi=1$, $\eta=1$, $C=A$, $D=A$, $|A| <\sqrt{p}$, $|\G| \le p^{3/5 - o(1)}$.
    Let us use a trivial bound for $\sum_{x\in \G} (A \c A) (x) \ge |\G \cap (A-A)|$.
    Then by (\ref{f:sumsets_intersections02}) one has
$$
    |\G \cap (A-A)| \ll |S_\G (A-A)|^{1/6} |A|^{3/2} \log^{1/4} |A| \,.
$$
    Thus it should be
    %$|\G|^6 |A|^{-9} \le |S_\G (A-A)| \le |A|^{3-o(1)}$
    $|S_\G (A-A)| \le |A|^{3-o(1)}$
    to obtain a non--trivial bound for the intersection.
    The quantity $|A|^3$ is some kind of a barrier for usefulness of  our bounds.
\end{example}

The arguments of the proof of Proposition \ref{p:sumsets_intersections} give us a general statement about the connection of $\T(A)$ and the
%ratio
product set/ratio set
of popular difference sets.

\begin{proposition}
    Let $\Gr$ be an abelian group and $A\subseteq \Gr$ be a set.
    Then
\begin{equation}\label{f:T(A)_general}
    \T(A) |A|^2 \min \{ |PP|, |P/P| \}  \gg \left( \sum_{x\in P} (A\c A) (x) \right)^4 \,.
\end{equation}
    Finally,
\begin{equation}\label{f:T(A)_general_2}
    \T^* (A) \le |A-A| \sum_{x} (A\c A)^3 (x) \,.
\end{equation}
\label{p:T(A)_general}
\end{proposition}
\begin{proof}
    Put $\sigma :=  \sum_{x\in P} (A\c A) (x)$.
    % and $L=\log |A|$.
    As in the proof of Proposition \ref{p:sumsets_intersections} we find a set $\t{A} \subseteq A$ such that
    for any $a\in \t{A}$ there exist $A_a \subseteq A$, $A_a-a\subseteq P$, $|A_a| \ge 2^{-1} \sigma |A|^{-1}$ and
\begin{equation}\label{tmp:03.05.2015_2}
    \sum_{a\in \t{A}} |A_a| \ge 2^{-1} \sigma \,.
\end{equation}
    For any $a,b\in \t{A}$, we have by the Cauchy--Schwarz inequality
$$
    \E^\times (A-a, A-b) \ge \E^\times (A_a - a, A_b-b) \ge |A_a|^2 |A_b|^2 / \min \{ |PP|, |P/P| \} \,.
$$
    Summing the last bound over all $a,b\in \t{A}$, we obtain in view of (\ref{tmp:03.05.2015_2})
    and
    %by
    the Cauchy--Schwarz inequality
    %again
$$
    \T(A) \min \{ |PP|, |P/P| \} \ge \left( \sum_{a\in \t{A}} |A_a|^2 \right)^2
        \gg
            \sigma^4 |\t{A}|^{-2}
                \ge
                    \sigma^4 |A|^{-2} \,.
$$

    To prove (\ref{f:T(A)_general_2}) just combine (\ref{f:T_square}) and the Cauchy--Schwarz inequality once more time
$$
    \T^* (A) \le \sum_{x, \la} \Cf^2_3 (A) (x,\la x) \cdot |A-A| = \sum_{x} (A\c A)^3 (x) \cdot |A-A| \,.
$$
    This completes the proof.
$\hfill\Box$
\end{proof}

\bigskip

\noindent{I.D.~Shkredov\\
%%\noindent{Отдел алгебры and теории чисел,\\
Steklov Mathematical Institute,\\
ul. Gubkina, 8, Moscow, Russia, 119991}
%MSU, IPPI RAN\\}
%%\\
%and
%\\
%Delone Laboratory of Discrete and Computational Geometry,\\
%Yaroslavl State University,\\
%Sovetskaya str. 14, Yaroslavl, Russia, 150000
\\
and
\\
IITP RAS,  \\
Bolshoy Karetny per. 19, Moscow, Russia, 127994\\
{\tt ilya.shkredov@gmail.com}

%\bigskip

%\no{E.V.~Solodkova\\
%IITP RAS,  \\
%Bolshoy Karetny per. 19, Moscow, Russia, 127994\\
%{\tt hsolodkova@gmail.com}}


\begin{thebibliography}{99}



\bibitem{BMR}
{\sc C. Bachoc, M. Matolcsi and I. Z. Ruzsa, }
{\em Squares and difference sets in finite fields, }
Integers 13 (2013), paper no. A77, 5 pp.


    \bibitem{Bourgain_more}
    {\sc J. Bourgain, }
    {\em More on the sum--product phenomenon in prime fields and its applications, }
    International Journal of Number Theory {\bf 1}:1 (2005), 1--32.



%\bibitem{BG}
%{\sc J.~Bourgain, M.Z.~Garaev, }
%\emph{On a variant of sum-product estimates and explicit exponential sum bounds in prime fields, }
%Math. Proc. Cambridge Philos. Soc. 146 (2009), no.1, 1--21.



%\bibitem{Rudin_book}
%{\sc W.~Rudin, } {\em Fourier analysis on groups,}  Wiley 1990 (reprint of the 1962 original).



\bibitem{DS_AD}
{\sc C. Dartyge and A. S\'{a}rk\"{o}zy, }
{\em On additive decompositions of the set of primitive roots modulo p, }
Monatsh. Math. {\bf 169} (2013), 317--328.



\bibitem{ES}
{\sc P. Erd\"{o}s, H. N. Shapiro, }
{\em On the least primitive roots of a prime, }
Pacific J. Math., 7:1 (1957), 861--865.


\bibitem{Green_sum-prod}
{\sc B. Green, }
{\em Sum--product phenomena in $\F_p$ : a brief introduction, }
2009, 10 pp., arXiv:0904.2075v1.


\bibitem{Jones_PhD}
{\sc T. G. F. Jones, }
{\em New quantitative estimates on the incidence geometry and growth of finite sets, }
PhD thesis, arXiv:1301.4853 (2013).


\bibitem{H-J}
{\sc R.~Horn, C.~Johnson, }
{\em Matrix Analysis, }
Cambridge University Press, Cambridge, 1985, xiii+561 pp.



\bibitem{Johnsen}
{\sc J. Johnsen, }
{\em On the distibution of powers in finite fields, }
J. Reine Angew. Math., 251, (1971), 10--19.


\bibitem{J_R-N}
{\sc T.G.F. Jones, O. Roche--Newton, }
{\em Improved bounds on the set $A(A+1)$, }
Journal of Combinatorial Theory, Series A 120.3 (2013), 515--526.


\bibitem{K_Tula}
{\sc S.~V.~Konyagin,}
{\em Estimates for trigonometric sums and for Gaussian sums, }
IV International conference "Modern problems of number theory and its applications". Part 3 (2002), 86--114.


    \bibitem{KS1}
    {\sc S.~V.~Konyagin, I.~Shparlinski, }
    {\em Character sums with exponential functions, } Cambridge University Press, Cambridge, 1999.


\bibitem{LS}
{\sc V.~F.~Lev, J.~Sonn, }
{\em Quadratic residues and difference sets, } arXiv:1502.06833.


%\bibitem{MV}


\bibitem{Mitkin}
{\sc D.~A.~Mit'kin,}
{\em Estimation of the total number of the rational points on a set of curves in a simple finite field,}
Chebyshevsky sbornik, {\bf 4}:4 (2003), 94--102.


\bibitem{R_Minkovski}
{\sc O.~Roche--Newton, }
{\em  A short proof of a near--optimal cardinality estimate for the product of a sum set, }
arXiv:1502.05560v1 [math.CO] 19 Feb 2015.


%\bibitem{R-N_R_S}
%{\sc O.~Roche--Newton, M.~Rudnev, I.D.~Shkredov, }
%New sum--product type estimates over finite fields // preprint.


%\bibitem{RN_Z}
%{\sc O.~Roche--Newton, D.~Zhelezov, }
%{\em A bound on the multiplicative energy of a sum set and extremal sum--products problems, }
%MJCNT, accepted.



\bibitem{Sarkozy_residues}
{\sc A. S\'{a}rk\"{o}zy,  }
{\em On additive decompositions of the set of the quadratic residues modulo p, }
Acta Arith. {\bf 155} (2012), 41--51.


\bibitem{SS1}
{\sc T. Schoen, I.D. Shkredov, }
{\em Higher moments of convolutions, }
J. Number Theory 133 (2013), no. 5, 1693--1737.


\bibitem{Sh_ineq}
{\sc I.D.~Shkredov, }
{\em Some new inequalities in additive combinatorics, }
Moscow J. Combin. Number Theory 3 (2013), 237--288.


\bibitem{s_mixed}
{\sc I.D. Shkredov, }
{\em Some new results on higher energies, }
Transactions of MMS, 74:1 (2013), 35--73.


\bibitem{Sh_Sarkozy}
{\sc I.D.~Shkredov, }
{\em Sumsets in quadratic residues, }
Acta Arith., {\bf 164}:3 (2014), 221--244.



\bibitem{Sh_energy}
{\sc I.D.~Shkredov, }
{\em Energies and structure of additive sets, }
Electronic Journal of Combinatorics, 21(3) (2014), \#P3.44, 1--53.



\bibitem{Sh_average}
{\sc I.D.~Shkredov, }
{\em On exponential sums over multiplicative subgroups of medium size, }
Finite Fields and Their Applications {\bf 30} (2014), 72--87.


\bibitem{Sh_3G}
{\sc I.D.~Shkredov, }
{\em On tripling constant of multiplicative subgroups, }
arXiv:1504.04522v1.


\bibitem{SSV}
{\sc I. Shkredov, E. Solodkova, I. Vyugin, }
{\em Intersections of multiplicative subgroups and Heilbronn's exponential sum, }
preprint.


\bibitem{Shparlinski_AD}
{\sc I.E. Shparlinski, }
{\em Additive decompositions of subgroups of finite fields, }
SIAM J. Discrete Math. {\bf 27} (2013), 1870--1879.



\bibitem{V-S}
{\sc I.V.~Vyugin, I.D.~Shkredov,}
{\em On additive shifts of multiplicative subgroups, }
Math. Sbornik. 203:{\bf 6} (2012), 81--100.


\bibitem{TV}
{\sc T. Tao and V. Vu, }
{\em Additive Combinatorics, }
Cambridge University Press (2006).


\bibitem{Z_A(A+1)}
{\sc D. Zhelezov, }
{\em On additive shifts of multiplicative almost--subgroups in $\F_p$, }
preprint.


%\bibitem{U}




\end{thebibliography}
\end{document}